\documentclass{amsart}[14pt]
\usepackage{amsmath, amstext,  amssymb, amsopn, amsthm, amsbsy, amsgen, amscd, amsfonts, times, mathptmx, epsfig}

\usepackage{color}

\usepackage{graphicx}
\usepackage{epstopdf}
\definecolor{CadetBlue}{cmyk}{0.62, 0.57, 0.23, 0 }
\definecolor{RoyalBlue}{cmyk}{1, 0.5, 0, 0 }

\definecolor{RedViolet}{cmyk}{0.07, 0.9, 0, 0.34 }
\definecolor{SeaGreen}{cmyk}{0.69, 0, 0.5, 0}
\definecolor{BrickRed}{cmyk}{0, 0.89, 0.94, 0.28}

\newcommand{\A}{\mathbb A}
\newcommand{\R}{\mathbb R}
\newcommand{\C}{\mathbb C}
\newcommand{\D}{\mathbb D}
\newcommand{\HP}{\mathbb H}

\newcommand{\N}{\mathbb N}
\newcommand{\PR}{\mathbb P}
\newcommand{\Q}{\mathbb Q}
\newcommand{\Z}{\mathbb Z}
\newcommand{\SO}{\hat{\mathbb S}}
\newcommand{\T}{\mathbb T}

\newcommand{\HSO}{\hat{\mathfrak{S}}}

\newtheorem{theo}{Theorem}

\newtheorem{prop}{Proposition}

\theoremstyle{definition}

\theoremstyle{remark}

\newtheorem{note}{Note}

\title[Notes on Nonlinear Number Fields]{Notes on Nonlinear Number Fields}
\author{T.M. Gendron \& A. Verjovsky}
\address{Instituto de Matem\'{a}ticas -- Unidad Cuernavaca, Universidad
Nacional Autonoma de M\'{e}xico, Av. Universidad S/N, C.P. 62210
Cuernavaca, Morelos, M\'{E}XICO}
\email{tim@matcuer.unam.mx, alberto@matcuer.unam.mx}
\date{1 July 2010}
\begin{document}
\vspace{2cm} \maketitle
\begin{abstract} 
This is a guide to the construction of a nonlinear number field, which includes some new points not found in \cite{GV}.     \end{abstract}
\tableofcontents

\section{Introduction}

These notes are intended to serve as a guide to the construction of nonlinear number fields, providing new observations not covered in our article \cite{GV}.  

This effort has been occasioned by the discovery of some errors in the published version \cite{GVp}\footnote{Some of which were brought to our attention by M. McQuillan, who we thank.  See the Errata \cite{GVe}.}.   All of these errors have since been corrected in the revised version of our article \cite{GV}, in such a way that the original aims found in \cite{GVp} remain as they were.  Nevertheless, we felt that it would be useful to make available a condensed and clear account which could companion the more detailed \cite{GV} (and to which we refer for a number of proofs).

In addition, we have taken this opportunity to incorporate new remarks which reinforce the aims of the original paper and tie up the notion of a nonlinear number field with classical number theoretic constructs. 

There is not a lot in the way of new mathematics contained here; rather, it is the point of view which is novel.  If there is virtue to be found in the notion of a nonlinear number field, it comes from a rethinking and reorganization of certain number theoretic ideas under a unified heading.  

For the record, we highlight straightaway what we feel are the essential features of the nonlinear number field ${\sf N}[K]$ associated to an algebraic extension $K/\Q$.  

\vspace{3mm}

\noindent 1.  {\bf {\em Character Group as a Field} (\S 3).} Let $\SO_{K}$ be the adele class group of $K$, ${\rm Char}(\SO_{K}) $ its group of characters.  Although the isomorphism of Pontryagin duality 
$ (K,+)\cong {\rm Char}(\SO_{K}) $
has been known since the time of Tate's landmark thesis \cite{Ta}, it seems that
the device of pushing forward the product of $K$ to
${\rm Char}(\SO_{K})$ has not previously been exploited.  This is crucial as it allows us to view ${\sf N}[K]$ -- a space of projective classes of holomorphic functions on a hyperbolized adele class group
$\hat{\mathfrak{S}}_{K}$, equipped with the Cauchy and Dirichlet products  -- as a kind of generalized field extension of $K$.

\vspace{3mm}

\noindent 2.   {\bf {\em Idele Class Group as a Galois Group} (\S 7).}  The Galois group ${\rm Gal}(L/K)$ acts by projective-unitary  maps
on ${\sf N}[L]$ fixing $ {\sf N}[K]$.   For $K=\Q$ there are a pair of flows
on a closure $\bar{\sf N}[\Q^{\rm ab}]$ of ${\sf N}[\Q^{\rm ab}]$
 preserving the Cauchy resp.\ the Dirichlet products and acting trivially on $\Q$, which yield monomorphisms of the idele class group ${\sf C}_{\Q}$ into
${\sf Gal}_{\oplus}(\bar{\sf N}[\Q^{\rm ab}]/\Q)$
and 
${\sf Gal}_{\otimes}(\bar{\sf N}[\Q^{\rm ab}]/\Q)$, the ``decoupled Galois
groups'' of automorphisms preserving $\oplus$ resp.\ $\otimes$ only.  The Dirichlet flow on
$\bar{\sf N}(\N )\subset \bar{\sf N}(\Q )$ is reminiscent of Berry's hypothetical Riemann flow.

 \vspace{3mm}

\noindent 3. {\bf {\em L-functions and Modular Forms as Nonlinear Numbers} (\S 8).}
The hyperbolized adele class group $\hat{\mathfrak{S}}_{K}$ considered here makes it possible to view spaces of (Hilbert) modular forms and (multi) L-functions, with their algebra intact, as subspaces of 
${\sf N}[K]$.  In particular, the Riemann zeta function defines a nonlinear rational integer and the Hecke algebra is subsumed by the Dirichlet product algebra
supplemented by a projection.

\vspace{3mm}

\noindent 4.   {\bf {\em Actions by Representation Categories} (\S 9).}  There is an action of ${\rm Char}({\sf C}_{\Q})$ by automorphisms on the {\it aperiodic unit class group} $\mathfrak{D}^{\star}_{\rm ap}[\N]$: the group of Dirichlet units whose Fourier coefficients are multiplicative and indexed by $\N$, modulo units which are periodic with respect to the Dirichlet flow.   This gives an intrinsic manifestation of the dual idele class group on the ground nonlinear field ${\sf N}[\Q]$.  More generally, there are actions by (L-functions of) 
Galois representations and (L-functions of)
cuspidal automorphic representations by endomorphisms of $\mathfrak{D}^{\star}_{\rm ap}[\N]$.  In particular, 
there is a functor ${\rm Rep}({\rm Gal}(\bar{\Q}/\Q))\rightarrow {\rm End}(\mathfrak{D}^{\star}_{\rm ap}[\N])$ 
which, modulo a verified global Langlands correspondence, gives a representation of monoidal categories.

\vspace{3mm}

\section{Geometry of the Adele Class Group}

\vspace{4mm}

\begin{center}
\small{\bf 2.1 Places }
\end{center}

\vspace{2mm}

\noindent Let $K$ be a number field of degree $d$ over $\Q$ and $O_{K}$ its ring of
integers.  For any place $\nu$
denote by $K_{\nu}$ the corresponding completion, for which there is a canonical
embedding $K\hookrightarrow K_{\nu}$.  Each prime ideal $\mathfrak{p}\subset O_{K}$ defines a finite place.  There are $d=r+2s$ infinite places, where
$r$ = number of real places, $2s$ = the number of complex places.

\vspace{4mm}

\begin{center}
\small{\bf 2.2 Minkowski Torus (\cite{Ne}, Chapter I.5) }
\end{center}

\vspace{2mm}

\noindent  Let $\C_{K}=\C^{d}$, where we index the coordinates of vectors by the $d$ infinite places.  Define
\[ K_{\infty}=\{ (z_{\nu})\in\C_{K}|\; \bar{z}_{\nu}=z_{\bar{\nu}}\} \cong\R^{r}\times\C^{s}. \] 
An inner-product on $K_{\infty}$ is obtained by restriction of the usual hermitian inner product
on $\C_{K}$.  There is a canonical embedding $K\hookrightarrow K_{\infty}$ given
by the place embeddings.  The image of $O_{K}$ is a lattice and the quotient
\[  \T_{K} = K_{\infty}/O_{K} \]
is called the \begin{color}{RoyalBlue} {\bf {\em Minkowski torus}}\end{color}.

\vspace{4mm}

\begin{center}
\small{\bf 2.3 Adele Class Group (\cite{Ko}, Chapter 1.5.3)}
\end{center}

\vspace{2mm}

\noindent The ring of \begin{color}{RoyalBlue} {\bf {\em finite adeles}}\end{color} is the restricted
product
\[  \A_{K}^{\rm fin}= \sideset{}{'}\prod_{\nu\;\text{finite}}K_{\nu} ,\]
restricted along the non-archimedean local rings of integers $O_{\nu}$.  The ring of \begin{color}{RoyalBlue}{\bf {\em adeles}}\end{color}
is the product
$\A_{K}= K_{\infty}\times \A_{K}^{\rm fin} $.
There is a canonical inclusion $K\hookrightarrow \A_{K}$ given by the place embeddings, and its image is discrete and co-compact.
The quotient
\[ \SO_{K}=\A_{K}/(K,+)\]
is called the \begin{color}{RoyalBlue}{\bf {\em adele class group}}\end{color}.   
Note that $ \SO_{K}$ is a $K$-vector
space and $\SO_{\Q}\otimes_{\Q} K \cong \SO_{K}$.

\vspace{4mm}

\begin{center}
\small{\bf 2.4 Solenoid Geometry of Adele Class Group (\cite{Ko}, Chapter 1.5.3;
\cite{GV}, \S 3)}
\end{center}

\vspace{2mm}

\noindent Denote by $\T_{\mathfrak{a}}=K_{\infty}/\mathfrak{a}$ where $\mathfrak{a}\subset O_{K}$ is an ideal.  We have an inverse system of tori and
\[   \SO_{K} \cong \lim_{\longleftarrow}\T_{\mathfrak{a}}. \]
Thus $\SO_{K}$ is an abelian pro Lie group of dimension $d$.  
Alternatively, we have
\begin{equation}\label{twistprod}
 \SO_{K}\cong (K_{\infty}\times\hat{O}_{K}) /O_{K} 
 \end{equation}
where $\hat{O}_{K}= \underset{\longleftarrow}{\lim} O_{K}/\mathfrak{a}$ is the $\mathfrak{a}$-adic completion of $O_{K}$, and where $O_{K}$ acts on $K_{\infty}\times\hat{O}_{K}$
 by  $\alpha \cdot ({\bf z}, \hat{\gamma} )=({\bf z}+\alpha, \hat{\gamma} -\alpha)$.

The presentation (\ref{twistprod}) shows that $\SO_{K}$
is a Cantor bundle over $\T_{K}$ in which the fibers are cosets of
the the subgroup $\hat{O}_{K}\hookrightarrow \SO_{K}$.
It also shows that $\SO_{K}$ is a  $d$-dimensional solenoid\footnote{A {\it solenoid} is a lamination having compact, totally disconnected transversals.}
whose leaves are the cosets of the dense subgroup $K_{\infty}\hookrightarrow \SO_{K}$. 
The inner-product on $K_{\infty}$ extends to a leaf-wise riemannian metric
on $\SO_{K}$.

\vspace{3mm}

\begin{center}
\small{\bf 2.5. Parabolic Adele Class Group}
\end{center}

\vspace{2mm}

The tensor product $\hat{\C}_{K} = \SO_{K}\otimes_{\Q} \C$ defines a solenoid
with leaves isomorphic to $\C_{K}=K_{\infty}\otimes \C $ which we call the
 \begin{color}{RoyalBlue} {\bf {\em parabolic adele class group}}\end{color}.  As in the previous section, we have
that $\hat{\C}_{K}$ is isomorphic to $(\C_{K}\times\hat{O}_{K})/O_{K}$ and to 
$\underset{\longleftarrow}{\lim} \;\C_{K}/\mathfrak{a}$.  Note that $\SO_{K}\subset \hat{\C}_{K} $ canonically, where it plays the role of the real axis.

\vspace{4mm}

\begin{center}
\small{\bf 2.6 Trace Map and Diagonal Inclusion (\cite{GV}, \S 3).}
\end{center}

\vspace{2mm}

\noindent Let $L/K$ be a finite extension where $L,K$ are of finite degree over $\Q$. The trace map ${\rm Tr}_{L/K}:L_{\infty}\rightarrow K_{\infty}$ -- defined
 by ${\bf y} = (y_{\mu})\mapsto {\bf x} = (x_{\nu})$ where 
 $x_{\nu} = \sum_{\mu|\nu} y_{\mu}$ -- extends to an epimorphism
\[ {\rm Tr}_{L/K}:\SO_{L}\longrightarrow \SO_{K}. \]  

The diagonal inclusion $i_{L/K}: K_{\infty}\rightarrow L_{\infty}$ -- defined
by $ {\bf x} = (x_{\nu})\mapsto {\bf y} = (y_{\mu})$ where
$y_{\mu} = x_{\nu}$ if $\mu|\nu$ -- extends to an embedding
\[ i_{L/K}:\SO_{K}\hookrightarrow \SO_{L}.\]  
If $d$ is the degree of $L/K$ then the scaled embedding $(1/d)i_{L/K}$ is a section of ${\rm Tr}_{L/K}$.

If $L/K$ is Galois
then the Galois group ${\rm Gal}(L/K)$ acts by isometric isomorphisms on
$\SO_{L}$ which leave invariant the trace map and the diagonal inclusion:
${\rm Tr}_{L/K}\circ \sigma = {\rm Tr}_{L/K}$ and
$\sigma\circ i_{L/K} = i_{L/K}$
 for all $\sigma\in  {\rm Gal}(L/K)$.

\vspace{4mm}

\begin{center}
\small{\bf 2.7  Proto Adele Class Group (\cite{GV}, \S 5)}
\end{center}

\vspace{2mm}

 \noindent Let $\mathcal{K}=\underset{\longrightarrow}{\lim} \, K_{\lambda}$
be an infinite degree algebraic number field (a direct limit of finite degree extensions of $\Q$).  Then the trace maps
induce an inverse limit of adele class groups and one defines
the  \begin{color}{RoyalBlue} {\bf {\em proto adele class group}}\end{color} of $\mathcal{K}$ as
\[  \hat{\SO}_{\mathcal{K}} =\lim_{\longleftarrow}\SO_{K_{\lambda}}.\]
This is a compact connected abelian group of infinite dimension.  
It is also a solenoid whose leaves are cosets of the locally convex but not locally compact vector space
\[  \hat{\mathcal{K}}_{\infty} := \lim_{\longleftarrow} (K_{\lambda})_{\infty}\hookrightarrow  \hat{\SO}_{\mathcal{K}}. \] 
If we take $\mathcal{K}=\bar{\Q}$, the associated proto adele class group $\hat{\SO}_{\bar{\Q}}$ is universal in the sense that it projects onto every
(proto) adele class group.

\vspace{4mm}

\begin{center}
\small{\bf 2.8  Hilbertian Adele Class Group (\cite{GV}, \S 4)}
\end{center}

\vspace{2mm}

\noindent The direct limit by the diagonal inclusions $\underset{\longrightarrow}{\lim}(K_{\lambda})_{\infty}$ can be given an inner-product
by rescaling, for each $\lambda$, the inner-product of $(K_{\lambda})_{\infty}$ by $1/d_{\lambda}$, 
where $d_{\lambda}$ is the degree of $K_{\lambda}/\Q$.  The completion is a Hilbert space denoted
$\mathcal{K}_{\infty}$.  The ring of integers $O_{\mathcal{K}}$ acts by isometries on $\mathcal{K}_{\infty}$ and
so we may define the  \begin{color}{RoyalBlue} {\bf {\em hilbertian adele class group}}\end{color}
\[  \SO_{ \mathcal{K}} = (\mathcal{K}_{\infty}\times \hat{O}_{\mathcal{K}})/O_{\mathcal{K}} \]
where $\hat{O}_{\mathcal{K}}$ is the pro-completion with respect to the
ideal quotients.
Then $\SO_{ \mathcal{K}}$ is a hilbertian lamination, and there is an embedding
$\SO_{ \mathcal{K}} \hookrightarrow  \hat{\SO}_{\mathcal{K}}$ with dense image.  See Theorem 4 of \cite{GV}.

\begin{center}
\small{\bf 2.9  Foliations of Adele Class Groups}
\end{center}

\vspace{2mm}

\noindent  We remark that for any extension of algebraic number fields $L/K$, cosets of  
$ i_{L/K}(K_{\infty})$  ``foliate" $\SO_{L}$ by planes of dimension
= $[K:\Q]$.  This secondary foliated structure is invariant by the action of the Galois group
when $L/K$ is Galois.
In particular, every (proto) solenoid has a canonical foliation by lines corresponding
to cosets of the image of $\Q_{\infty}=\R$ by the diagonal embedding.

\pagebreak

\section{Character Field }

\vspace{4mm}

\begin{center}
\small{\bf 3.1 Standard Character}
\end{center}

\vspace{2mm}

\noindent Fix a number field $K/\Q$ of finite degree $d$.  We have a canonical 
``trace functional"
$  {\sf Tr}={\rm Tr}_{K/\Q}: K_{\infty}\longrightarrow \Q_{\infty}=\R $
given by
${\bf z}= (z_{\nu})\mapsto \sum z_{\nu}$.
Using this we define a character on $K_{\infty}$ by
\[  \psi_{\infty}^{K}({\bf z})= \exp (2\pi i  \, {\sf Tr}({\bf z})), \]
which extends continuously to
a character $\psi_{K}$ of
$\SO_{K}$.   We call this the  \begin{color}{RoyalBlue} {\bf {\em standard character}}\end{color}.  See \cite{GV}, proof of Theorem 5.
This choice of character is natural with respect to pull-back along trace maps: $\psi_{L}= \psi_{K}\circ {\rm Tr}_{L/K}$.

\vspace{4mm}

\begin{center}
\small{\bf 3.2 Another Description of the Standard Character (\cite{Bu}, Chapter 3.1)}
\end{center}

\vspace{2mm}

\noindent Alternatively, one can 
define $\psi_{K}$ on $\SO_{K}$ by pushing down the following character
of  $\A_{K}$ that is trivial on $K$:
\[ \psi_{\infty}^{K}\times \prod_{\nu \; \text{finite}} \psi^{K}_{\nu}.\]
Here, the local component  $\psi^{K}_{\nu}$ is defined like so.  Let
${\sf Tr}_{\nu}:K_{\nu}\rightarrow\Q_{p}$ be the local trace map.
Every element
of $\alpha_{\nu}\in K_{\nu}$ may be written in the form $q_{\nu}+o_{\nu}$ where ${\sf Tr}_{\nu}(q_{\nu})\in \Z [1/p]$ and ${\sf Tr}_{\nu}(o_{\nu})\in\Z_{p}$. Then one
defines 
\[ \psi^{K}_{\nu}(\alpha_{\nu}) := \exp (-2\pi i \, {\sf Tr}_{\nu}(q_{\nu})).\]
The induced character on $\SO_{K}$ is then equal to $\psi_{K}$.  See 
Exercise 3.1.2, pg. 277 of \cite{Bu}.

\vspace{4mm}

\begin{center}
\small{\bf 3.3 Character Field (\cite{GV}, \S 7)}
\end{center}

\vspace{2mm}

\noindent  \begin{theo}\label{chartheo}  There is a natural isomorphism of $(K,+)$ with the 
character group ${\rm Char}(\SO_{K})$.  Moreover, the product on $K$ may be
pushed forward to ${\rm Char}(\SO_{K})$ in a way which is natural with respect
to field extensions and pull-backs of characters along trace maps.  Thus 
${\rm Char}(\SO_{K})$ is in a natural way a {\rm field}.
\end{theo}

The first statement is 
classical e.g. see \cite{Ta}, Theorem 4.1.4.  The isomorphism is defined by $\alpha\mapsto \psi_{\alpha}^{K}$
where $\psi_{\alpha}^{K}(\hat{{\bf z}})= \psi_{K}(\alpha\hat{{\bf z}})$. Note that
the expression $\alpha\hat{{\bf z}}$ makes sense since $\SO_{K}$ is a $K$-vector space.   Naturality is an expression of the commutative diagram
\[ 
\begin{CD}
L    @>\cong>>    {\rm Char}(\SO_{L})\\
\cup @.                 @AA{\rm Tr}_{L/K}^{\ast}A\\
K    @>\cong>>    {\rm Char}(\SO_{K})
\end{CD}
\]
whose commutativity follows from the fact that for all $\alpha\in K$,
\[ {\rm Tr}_{L/K}^{\ast} (\psi^{K}_{\alpha}) = \psi^{K}_{\alpha}\circ  {\rm Tr}_{L/K} =
\psi^{L}_{\alpha}.
  \]
(Dually, Pontryagin duality shows that the inclusion ${\rm Tr}_{L/K}^{\ast}:{\rm Char}(\SO_{K}) \hookrightarrow {\rm Char}(\SO_{L})$ induces the trace ${\rm Tr}_{L/K}:\SO_{L}\rightarrow \SO_{K}$.)

Due to this Theorem, we will refer to ${\rm Char}(\SO_{K})$ as the  \begin{color}{RoyalBlue} {\bf {\em character field}}\end{color} of $\SO_{K}$, and we will denote by $\oplus$ resp.\ $\otimes$ the operations on characters corresponding to $+$ resp.\ $\times$.  In this isomorphism, the character group of the
Minkowski torus ${\rm Char}(\T_{K})\subset {\rm Char}(\SO_{K})$ 
is identified with the inverse different $\mathfrak{d}_{K}^{-1}\supset O_{K}$,
an $O_{K}$-module extension of $O_{K}$.  See \cite{GV}, Theorem 5.

The naturality part of the Theorem implies that if
$\mathcal{K}=\underset{\longrightarrow}{\lim}K_{\lambda}$ is an infinite degree field
extension then
\[   {\rm Char}(\hat{\SO}_{\mathcal{K}})\cong \lim_{\longrightarrow}  
{\rm Char}(\SO_{K_{\lambda}})\cong\mathcal{K}, \]
and so is also a field.    The character field ${\rm Char}(\hat{\SO}_{\bar{\Q}})$ is thus available to play its expected role as universal object.

\section{Nonlinear Number Fields I: Field Algebras}

While our notation in this section will be 
that which we have reserved for finite field extensions, everything discussed here extends without modification to the case of an infinite degree algebraic extension $\mathcal{K}/\Q$.

\vspace{4mm}

\begin{center}
\small{\bf 4.1 Field Algebra (\cite{GV}, \S 9)}
\end{center}

\vspace{2mm}

\noindent Let $\C [K]$ be the  \begin{color}{RoyalBlue} {\bf {\em field algebra}}\end{color} associated to $K$.  By this we mean
the $\C$-vector space of formal polynomials 
\[ f=\sum_{\alpha\in K} a_{\alpha}\cdot \alpha\] with coefficients in $\C$ (where all but a finite number of the $a_{\alpha}$ are zero), equipped with
operations $\oplus$, $\otimes$ induced
from $+$, $\times$.
In order not to confuse the vector space identity $0$ with the additive identity in $K$, we denote by $1_{\oplus}$, $1_{\otimes}$ the identity for $+$, resp. $\times$.

If $f=\sum a_{\alpha}\cdot \alpha$ , $g=\sum b_{\alpha}\cdot \alpha$ belong to  $\C [K]$ then we have
\[  f\oplus g = \sum_{\alpha}\bigg(\sum_{\alpha=\gamma_{1}+\gamma_{2}} a_{\gamma_{1}}b_{\gamma_{2}}\bigg)\alpha, \quad\quad
 f\otimes g = \sum_{\alpha}\bigg(\sum_{\alpha=\gamma_{1} \gamma_{2}} a_{\gamma_{1}}b_{\gamma_{2}}\bigg)\alpha\]
 and we refer to these as the  \begin{color}{RoyalBlue} {\bf {\em Cauchy product}}\end{color} resp.\ 
  \begin{color}{RoyalBlue} {\bf {\em Dirichlet product}}\end{color}
 of $f$ and $g$.  We note that the Dirichlet product does not distribute over
 the Cauchy product.  There is a canonical monomorphism $K\hookrightarrow \C [K]$
 given by the monomials.
 
\vspace{4mm}

\begin{center}
\small{\bf 4.2 Pre-Nonlinear Number Field (\cite{GV}, \S 9)}
\end{center}

\vspace{2mm}

\noindent  Let ${\sf T}: \C [K]\longrightarrow \C$ be defined 
 $  {\sf T}(f)= \sum a_{\alpha}\in\C $.
 This map takes both of the operations  $\oplus$, $\otimes$ to the product in $\C$:
 \begin{equation}\label{trace} {\sf T}(f\oplus g)=  {\sf T}(f\otimes g) =  {\sf T}(f) {\sf T}(g).   
 \end{equation}
It follows that if we write ${\sf Z}[K]= {\rm Ker}({\sf T})$ then the {\it set-theoretic difference} of vector spaces
\[  \C [K] -{\sf Z}[K] \]
 is invariant with respect to both operations $\oplus$
and $\otimes$.  We can projectivize this space obtaining an {\it infinite dimensional
affine subspace }  ${\sf N}_{0}[K]\subset\PR\C [K]$ on which both
$\oplus$, $\otimes$ descend.  This space is called the  \begin{color}{RoyalBlue} {\bf {\em pre-nonlinear number field}}\end{color} of $K$.  It is a double semi-group and is the precursor
to the nonlinear number field ${\sf N}[K]$.

\pagebreak

\vspace{4mm}

\begin{center}
\small{\bf 4.3 Affine Structure of ${\sf N}_{0}[K]$}
\end{center}

\vspace{2mm}

\noindent 

\noindent Since ${\sf N}_{0}[K]$ is the complement of
a codimension 1 projective subspace (the projectivization $\PR{\sf Z}[K]$ of 
${\sf Z}[K]$),
it is an affine subspace.  In fact, it may be canonically identified with the affine hyperplane in $\C [K]$
\[ {\sf Z}_{1}[K]: = 
\{Êf |\; {\sf T}(f)=1 \} .\]  Explicitly if $f\in[f]\in {\sf N}_{0}[K]$ then the map $[f]\mapsto (1/{\sf T}(f))\cdot f$ identifies ${\sf N}_{0}[K]$ with
$ {\sf Z}_{1}[K]$.  Both Cauchy and Dirichlet multiplication restrict to ${\sf Z}_{1}[K]$ by equation
(\ref{trace}), and the identification ${\sf N}_{0}[K]\rightarrow {\sf Z}_{1}[K]$
is an isomorphism of double semigroups.
 Note that ${\sf Z}_{1}[K]$ contains the canonical basis 
$K$ of $ \C [K]$.

\vspace{4mm}

\begin{center}
\small{\bf 4.4 Vector Space Structure of ${\sf N}_{0}[K]$}
\end{center}

\vspace{2mm}

\noindent Choosing the point $1_{\oplus}\in {\sf Z}_{1}[K]$ as an origin,
${\sf Z}_{1}[K]$ becomes in the usual way a vector space.   
Explicitly, the sum and scalar product are \[   f \dot{+} g := f+g-1_{\oplus}, \quad  c\odot f := cf + (1-c)1_{\oplus}.\]
Note that
Cauchy resp.\ Dirichlet multiplication act affine linearly resp.\ linearly in ${\sf Z}_{1}[K]$.   Indeed, 
$\oplus$ affinely respects the sum:
\[ \big( h\oplus ( f \dot{+} g ) \big)\dot{+} h= h\oplus (  f+g-1_{\oplus} ) +h-1_{\oplus}= (h\oplus f) + (h\oplus g) - 1_{\oplus}=
(h\oplus f )\dot{+}  ( h\oplus g ).\]
And $\oplus$ affinely respects scalar multiplication:
\[ \big( h\oplus (c\odot f)\big)    \dot{+}   \big( (c-1)h +1_{\oplus}  \big)=   c\odot (h\oplus f) .  \]
The linearity of Dirichlet multiplication is likewise verified, using the fact that $f\otimes 1_{\oplus}=1_{\oplus}$ for all $f\in {\sf Z}_{1}[K]$.

\vspace{4mm}

\begin{center}
\small{\bf 4.5 Galois Action}
\end{center}

\vspace{2mm}

\noindent Let $L/K$ be an extension so that ${\sf N}_{0}[K]\subset 
{\sf N}_{0}[L]$.
If $L/K$ is Galois then $\sigma\in {\rm Gal}(L/K)$ defines
a field algebra automorphism of $\C [L]$ via 
\[ f\mapsto \sigma (f) = \sum a_{\alpha}\sigma (\alpha)=  \sum a_{\sigma^{-1}(\alpha)}\alpha .\]
Indeed,
\[  \sigma (f\oplus g ) =   
\sum_{\alpha}\bigg(\sum_{\sigma^{-1}(\alpha)=\gamma_{1}+
\gamma_{2}} a_{\gamma_{1}}b_{\gamma_{2}}\bigg)\alpha =
\sum_{\alpha}
\bigg(\sum_{\alpha=\gamma_{1}+\gamma_{2}} a_{\sigma^{-1}(\gamma_{1}) }b_{\sigma^{-1}(\gamma_{2})}\bigg)\alpha =
 \sigma (f)\oplus \sigma (g).  \]
One shows identically that $\sigma (f\otimes g )=\sigma (f)\otimes \sigma (g)$.   
Note further that  ${\rm Gal}(L/K)$ preserves ${\sf Z}_{1}[L]$ and
acts by automorphisms with respect to its vector space structure, since
$\sigma (1_{\oplus}) = 1_{\oplus}$.   By projectivization we induce an action of ${\rm Gal}(L/K)$
on ${\sf N}_{0}[L]$ preserving its double semi group and vector space structures.

\section{Nonlinear Number Fields II: Completions}

We extend the discussion of the previous
section to square summable functions, in order to introduce Hilbert space techniques 
and standard ideas of harmonic analysis.

\pagebreak

\vspace{4mm}

\begin{center}
\small{\bf 5.1 Puiseux Representation (\cite{GV}, \S 9)}
\end{center}

\vspace{2mm}

\noindent We shall henceforward interpret the elements of $\C[K]$ as functions
on the solenoid $\SO_{K}$ via Pontryagin duality: that
is, we replace $\alpha$ by the character $\psi_{\alpha}(\hat{\bf z})$
and write the element $f=\sum a_{\alpha}\alpha$ as
\[  f (\hat{\bf z}) = \sum a_{\alpha}  \psi_{\alpha}(\hat{\bf z}). \]
We will often consider $f$ on the dense leaf $K_{\infty}$ where
it takes the form of a $K$-Puiseux polynomial:
\[ f ({\bf z}) =\sum a_{\alpha} \exp (2\pi i {\sf Tr}(\alpha{\bf z})) = 
\sum a_{\alpha} \xi^{\alpha}, \quad \;\;\xi :=  \exp (2\pi i {\sf Tr}({\bf z}))  .\] 
This observation is interesting in its own right e.g. since Puiseux series are used to
parametrize complex singularities.   Notice that in the Puiseux representation, $1_{\otimes}=\xi$.

\begin{note} The space of continuous functions $C_{0}(\SO_{K},\C )$ on $\SO_{K}$ can be identified with the space of {\it limit periodic functions}. It is worth mentioning that by a Theorem of Pontryagin \cite{Be},  $\SO_{K}$ itself can be obtained as the convex hull of a certain limit periodic function of $K_{\infty}$.  The inclusion 
$\C [K]\subset C_{0}(\SO_{K},\C )$ is dense.   
\end{note}

\vspace{4mm}

\begin{center}
\small{\bf 5.2 $\;L^{2}$ Completion (\cite{GV}, \S 9)}
\end{center}

\vspace{2mm} 

\noindent 

Using the Haar probability measure of $\SO_{K}$, we may
define the usual $L^{p}$ norms on $C_{0}(\SO_{K},\C )$ and speak in particular of the completions
with respect to them, which will be denoted $L^{p}[K]$.
The completion $L^{1}[K]$ is the Fourier algebra of functions
on $\SO_{K}$.  One could also complete $\C [K]$ using the norm topology (to get a  $C^{\ast}$-algebra), or using a weak operator topology (to get a von Neumann algebra), however we shall not pursue these completions here.  
Our interest
will be to extend the construction of the previous section to $L^{2}[K]$.  

Note that by Fourier theory, we may identify $L^{2}[K]$ with $l^{2}[K]$ =  the space
of square summable functions $K\rightarrow\C$.  In addition, the field
$K$, identified with the monomials $\xi^{\alpha}$, forms an orthonormal
basis of $L^{2}[K]$.   Thus every element $f\in L^{2}[K]$ may be written as
an $L^{2}$ Puiseux series
\[ f(\xi )=\sum a_{\alpha} \xi^{\alpha}. \]
If $g(\xi )= \sum b_{\alpha} \xi^{\alpha}$, its inner-product 
with $f$ is given by Parseval's formula 
\[  \langle f, g\rangle = \sum_{\alpha} a_{\alpha}\bar{b}_{\alpha} . \]
When $\mathcal{K}$ is of infinite degree, we use the Haar probability
measure on the compact proto adele class group $ \hat{\SO}_{\mathcal{K}}$ to define
$L^{2}[\mathcal{K}]$.

\vspace{4mm}

\begin{center}
\small{\bf 5.3 Definition of the Cauchy and Dirichlet Products}
\end{center}

\vspace{2mm}

For any $f\in L^{2}[K]$ and $\alpha\in K$ denote by $S_{\alpha}f$ the function in $L^{2}[K]$ with Fourier series
\[ S_{\alpha}f = \sum_{\beta\in K} a_{\alpha - \beta}\xi^{\beta}. \]
Thus $ S_{\alpha}f $ is the bilateral shift by $\alpha$ of the function $f(\bar{\bf z})$.
For $f,g\in L^{2}[K]$ define
\[ c_{\alpha}= 
\sum_{\alpha=\alpha_{1}+\alpha_{2}} a_{\alpha_{1}}b_{\alpha_{2}} . \]
Then $c_{\alpha}= \langle f, S_{\alpha}\bar{g}\rangle$ (where 
$\bar{g}=\sum \bar{b}_{\beta}\xi^{\beta}$) and so is an unambiguously defined
complex number.
We then say that the Cauchy product $f\oplus g:=\sum c_{\alpha}\xi^{\alpha}$ is defined if the defining series belongs to
 $L^{2}[K]$.  

One similarly treats the Dirichlet product using 
\[ T_{\alpha}f = \sum_{\beta\in K^{\ast}} a_{\alpha\beta^{-1}}\xi^{\beta}. \]
Then for $\alpha\not=0$ 
\[ d_{\alpha}=  \sum_{\alpha=\alpha_{1}\alpha_{2}} 
a_{\alpha_{1}}b_{\alpha_{2}}  \] 
is again well-defined since it is equal to $\langle f, T_{\alpha}\bar{g}\rangle$.
Consider also the formal expression
\[   d_{0} = a_{0}\sum_{\alpha\in K^{\ast}} b_{\alpha}  + b_{0}\sum_{\alpha\in K^{\ast}} a_{\alpha} 
+a_{0}b_{0}.
 \]
 Then we say that the Dirichlet product $f\otimes g:=\sum d_{\alpha}\xi^{\alpha}$ is defined when $d_{0}$ converges absolutely and the defining sum belongs to $L^{2}[K]$.

 As in functional
analysis we view the maps
$    g\mapsto f\oplus  g, \; g\mapsto f\otimes  g$
as (possibly unbounded) linear operators 
\[ M_{\oplus}^{f}: {\sf Dom}_{\oplus}(f)\longrightarrow L^{2}[K],\quad
    M_{\otimes}^{f}: {\sf Dom}_{\otimes}(f) \longrightarrow L^{2}[K]\] 
whose domains are dense subspaces of 
$L^{2}[K]$ since they contain $\C[K]$.  
Note that for any $f\in \C [K]$, 
${\sf Dom}_{\oplus}(f)={\sf Dom}_{\otimes}(f)=L^{2}[K]$.

\vspace{4mm}

\begin{center}
\small{\bf 5.4 Nonlinear Number Field (\cite{GV}, \S 9)}
\end{center}

\vspace{2mm} 

\noindent

Let $f\in L^{2}[K]$.  Define ${\sf T}(f)$ to be the sum $\sum a_{q}$ provided that the latter is absolutely convergent, and otherwise, ${\sf T}(f)$ is not defined.  It follows from the definition of the Dirichlet product in \S 5.3 that the set of $f$ for which the trace is defined
is precisely ${\sf Dom}_{\otimes}(1_{\oplus})$.   Let 
\[ \widetilde{\sf Z}[K]=\{f\in {\sf Dom}_{\otimes}(1_{\oplus})\, |\; {\sf T}(f)=0\} .\]
Note that $\widetilde{\sf Z}[K]$ is not the $L^{2}$ closure of ${\sf Z}[K]$, which is all of $L^{2}[K]$
(see \S 5.5 below).
 One then
defines the  \begin{color}{RoyalBlue} {\bf {\em nonlinear number field}}\end{color} associated to $K$ as
\[  {\sf N}[K] : = \PR (L^{2}[K]- \widetilde{\sf Z}[K])\subset  \PR L^{2}[K]. \]
The operations of $\oplus$ and $\otimes$ induce partially
defined operations in ${\sf N}[K]$, confined to appropriate domains of definition.
We refer to such a structure as a  \begin{color}{RoyalBlue} {\bf {\em partial double semigroup}}\end{color}.  

\vspace{4mm}

\begin{center}
\small{\bf 5.5 Wienerian Nonlinear Number Field}
\end{center}

\vspace{2mm} 

Notice that ${\sf Dom}_{\otimes}(1_{\oplus})\subset l^{1}[K]\cap l^{2}[K]$ is the Banach space of absolutely convergent series,
equipped with the norm $\|f\| = \sum |a_{\alpha}|$.
The Cauchy and Dirichlet products are globally
defined in ${\sf Dom}_{\otimes}(1_{\oplus})$ since we have the bound of norms
\[ \|f\oplus g\|, \|f\otimes g\|\leq \sum |a_{\alpha}||b_{\beta}|<\infty . \]
(The latter series is a rearrangement of the terms of the Cauchy product of the absolutely convergent
series $\sum |a_{\alpha}|, \sum |b_{\beta}|$, hence is
absolutely convergent \cite{Kn}.)  Thus $W[K]:={\sf Dom}_{\otimes}(1_{\oplus})$ is the \begin{color}{RoyalBlue} {\bf {\em Wiener field algebra}}\end{color}.
The projectivization 
\[  {\sf W}[K] : = \PR (W[K]- \widetilde{\sf Z})  \]
is a full-fledged semi group, the \begin{color}{RoyalBlue} {\bf {\em Wienerian nonlinear number field}}\end{color} of $K$.  In this context the analogue of Wiener's Theorem \cite{New} holds: namely
if $f(\xi )\not=0$ for all $\xi$ then the function $1/f\in W[K]$ and is the Cauchy inverse of $f$.

Analogues of Wiener's Theorem for the Dirichlet product were proved in
\cite{HW}, \cite{GN}.  The result contained in \cite{GN} implies that if $f\in {\sf W}[\Q_{+}]$ has associated Dirichlet series $D_{f}(s) = \sum a_{q}q^{-s}$ which is uniformly bounded away from
zero for ${\rm Re}(s)>0$ then $f$ has a Dirichlet inverse.

\vspace{4mm}

\begin{center}
\small{\bf 5.6 Full Projectivization}
\end{center}

\vspace{2mm} 

It will also be useful to consider the full projectivization of $L^{2}[K]$ which we denote
\[  \bar{\sf N}[K] = \PR (L^{2}[K]). \]
We have embeddings, the last three of which are dense:
\[ K\subset {\sf N}_{0}[K]\subset {\sf W}[K]  \subset {\sf N}[K] \subset \bar{\sf N}[K]. \]
The considerations of \S\S 5.1-5.6 apply without change to an infinite degree algebraic number field
$\mathcal{K}/\Q$.

\vspace{4mm}

\begin{center}
\small{\bf 5.7 Completion of ${\sf N}_{0}[K]$ as a Pre-Hilbert Space}
\end{center}

\vspace{2mm} 

\noindent Let us consider what would have happened if
we had taken as our definition of $ {\sf N}[K]$ the completion ${\sf N}_{0}[K]\cong
{\sf Z}_{1}[K]$ with respect to the $L^{2}$ inner-product
on ${\sf Z}_{1}[K]$.  It is not difficult to see that 
${\sf Z}[K]$ is dense in $L^{2}[K]$ (e.g.\ every element of the basis $K$
is an $L^{2}$-limit of elements of ${\sf Z}[K]$)
hence ${\sf Z}_{1}[K]$ is dense as well.  Thus the
completion of ${\sf N}_{0}[K]$ is isometric to $L^{2}[K]$ with 
its vector space operations
shifted to $\dot{+}$ and $\odot$.

\vspace{4mm}

\begin{center}
\small{\bf 5.8 Unitary Galois Representation}
\end{center}

\vspace{2mm} 

\noindent If $L/K$ is a finite degree extension then we have an
inclusion $L^{2}[K]\subset L^{2}[L]$.
If the extension is Galois, then the Galois action defines a faithful
infinite dimensional representation
\[  \rho: {\rm Gal}(L/K) \longrightarrow {\sf U}(L^{2}[L])   \]
into the group of unitary or anti-unitary operators of $L^{2}[L]$.
Since this action is the identity on the closed subspace $L^{2}[K]$
it induces a faithful unitary representation of the quotient Hilbert space
$L^{2}[L]/L^{2}[K]$, also infinite-dimensional.

If ${\rm Gal}(L/K)$ is abelian, then all of its irreducible representations have dimension 1 by Schur's Lemma, so that there is a basis of $L^{2}[L]$ consisting of
eigenvectors for the representation $\rho$, each of which is determined by a character.  In general, we conjecture that this representation contains all irreducible representations.  We emphasize
that this representation is not the right regular one, for simple dimensional reasons.

We also get Galois representations in the case of infinite algebraic extensions and in particular we obtain a faithful infinite dimensional unitary representation
of ${\rm Gal}(\bar{\Q}/\Q)$.
In the abelian case of ${\rm Gal}(\bar{\Q}^{\rm ab}/\Q)\cong\hat{\Z}^{\times}$ it would be interesting to determine the spectral decomposition of this representation and whether there is a common spectral measure.

The projectivization of this action preserves ${\sf N}_{0}[L]$, ${\sf W}[L]$ and $ {\sf N}[L]$ (since it preserves trace zero) and
reduces to the usual Galois action on $L\subset {\sf N}[L]$.

 \pagebreak

\vspace{4mm}

\begin{center}
\small{\bf 5.9 Multiplication Operators}
\end{center}

\vspace{2mm} 

\noindent  The Cauchy and Dirichlet multiplication operators 
$M_{\oplus}^{f}$, $M_{\otimes}^{f}$ are in general only partially defined,
however if $f=\gamma \in K^{\ast}$ then ${\sf Dom}_{\oplus}(\gamma )= 
{\sf Dom}_{\otimes}(\gamma ) = L^{2}[K]$ and the actions
are unitary since they act by bilateral shift on the sequence
space $l^{2}[K]$:
\[  \langle  M_{\oplus}^{\gamma}(f),\; M_{\oplus}^{\gamma}(g)\rangle =
\sum_{\alpha} a_{\alpha-\gamma}\bar{b}_{\alpha-\gamma} =
 \langle  f, g\rangle =
 \sum_{\alpha} a_{\gamma^{-1}\alpha}\bar{b}_{\gamma^{-1}\alpha} = 
 \langle  M_{\otimes}^{\gamma}(f),\; M_{\otimes}^{\gamma}(g)\rangle.
\]
The algebra $\mathcal{A}_{\otimes}$ generated by the operators $M^{f}_{\otimes}$ for $f\in \C[K]$, is commutative, and 
 we shall see in \S 8 that when augmented by a projection it contains 
the classical Hecke algebras.

\section{Nonlinear Number Fields III: Hardy Grading}

To simplify the exposition we will in this section restrict attention
to totally real fields $K/\Q$.  When $K$ has complex places, a more
general discussion must be made using the notion of {\it complex signs}: this is found in \cite{GV}, \S\S 6,8.

\vspace{4mm}

\begin{center}
\small{\bf 6.1 Hyperbolization (\cite{GV}, \S 6)}
\end{center}

\vspace{2mm} 

\noindent For each place $\nu$ of $K$ consider the half-space 
$\HP_{\nu}=K_{\nu}\times i(0,\infty )$ 
with its hyperbolic metric and
write
\[ \HP_{K}\;\; =\ \;\;  \prod \HP_{\nu} 
\;\; \cong \;\;K_{\infty}\times (0,\infty)^{d} ,\]
a $d$-dimensional 
complex polydisk equipped with
the product riemannian metric.
Points of $\HP_{K}$ are denoted 
${\boldsymbol \tau} = (\tau_{\nu})$,  $\tau_{\nu}=x_{\nu}+it_{\nu}$.
Note that $\HP_{K}\subset \C_{K}$.

The subgroups $O_{K}$ and $K$ of
$K_{\infty}$, viewed as translations, extend by isometries to
$\HP_{K}$.  The quotients
\begin{equation}\label{hyperbolized}
\mathfrak{T}_{K} \;\; =\;\;  \HP_{K}/O_{K}\;\;\approx\;\; 
(\Delta^{\ast})^{d}, \quad\quad\quad\quad
\hat{\mathfrak{S}}_{K}  \;\; =\;\;  
\Big( \HP_{K}\times \A^{\sf fin}_{K}\Big) \Big/ K \;\; \approx\;\;
\SO_{K}\times (0,\infty )^{d}
\end{equation}
are referred to as the  \begin{color}{RoyalBlue} {\bf {\em hyperbolic Minkowski torus}}\end{color} and the  \begin{color}{RoyalBlue} {\bf {\em hyperbolic adele
class group}}\end{color} 
of $K$.  Here $\Delta^{\ast}$ denotes the punctured hyperbolic disk.   We have
$\hat{\mathfrak{S}}_{K} \subset \hat{\C}_{K}$, and we note
that
the adele class group $\SO_{K}$ may be identified with the Shilov boundary
of $\hat{\mathfrak{S}}_{K}$, see \cite{ArSi}, \cite{ArSi1}.
These objects may be thought of as model (solenoidal) cusps.
Hyperbolizations of solenoids first appeared in the construction of Sullivan's solenoid \cite{Gh}, \cite{Su}.  

In the case of an infinite field extension $\mathcal{K}$, 
one follows the prescription of the preceding paragraph using the hilbertian
torus and solenoid $\T_{\mathcal{K}}$
and $\SO_{\mathcal{K}}$.   

\vspace{4mm}

\begin{center}
\small{\bf 6.2 Hardy Space (\cite{GV}, \S 8)}
\end{center}

\vspace{2mm} 

\noindent A continuous function $F:\HSO_{K}\rightarrow\C$
is holomorphic if its restriction to each leaf is
holomorphic, or equivalently, 
if its restriction to $\HP_{K}$ is holomorphic. 

For each ${\bf t}\in  (0,\infty )^{d}$ let 
$\SO_{K}({\bf t})\subset\HSO_{K}$ be the subspace of points having 
extended coordinate
${\bf t}$ with respect to the decomposition (\ref{hyperbolized}).  Since $\SO_{K}({\bf t})\approx\SO_{K}$ we may put
on $\SO_{K}({\bf t})$ the unit mass Haar measure and define 
for $F,G :\HSO_{K}\rightarrow \C$ the pairing
\[ (F,G)_{\bf t}\;\;=\;\ \int_{\SO_{K}({\bf t})} 
F\,\bar{G}\; d\mu . \]
 The  \begin{color}{RoyalBlue} {\bf {\em Hardy space}}\end{color} associated to $K/\Q$ 
is the Hilbert space
\[ {\sf H}[K]\;\;=\;\; \Big\{ F:\HSO_{K}\rightarrow\C\;\Big| 
\;\;F\text{ is holomorphic and } \sup_{\bf t} (F,F)_{\bf t}<\infty\Big\} \]
with inner-product 
$(\cdot ,\cdot )=\sup_{\bf t}(\cdot ,\cdot )_{\bf t}$.  

Any 
$F\in {\sf H}[K]$
has an {\em a.e.}\ defined $L^{2}$ limit for
${\bf t}\rightarrow {\bf 0}$ \cite{Ka}, defining an element 
$\partial F \in L^{2}[K]$.
Using the Fourier development
available there, we may write the restriction $F|_{\HP_{K}}$ as a holomorphic
Puiseux series: 
\begin{equation}\label{boundary}  F|_{\HP_{K}}({\boldsymbol \tau}) =\sum a_{\alpha}
\exp (2\pi i {\sf Tr}(\alpha\cdot {\boldsymbol \tau} ))
= \sum_{\alpha}a_{\alpha} 
{\boldsymbol \eta}^{\alpha},
\end{equation}
where ${\boldsymbol \eta}^{\alpha}:= \exp (2\pi i {\sf Tr}(\alpha\cdot {\boldsymbol \tau} ))$.

The positive cone $K_{\infty}^{+}\subset K_{\infty}$ is the set of ${\bf x} = (x_{\nu})\in K_{\infty}$ for which $x_{\nu}> 0$
for all $\nu$.  
It is clear then that for the series of (\ref{boundary}) to define elements of ${\sf H}[K]$,
we must have
$a_{\alpha}=0$ for $\alpha\not\in K_{\infty}^{+}$.
Note that 
$\| F\|^{2}=\sum |a_{\alpha}|^{2}$ and hence the correspondence $F\mapsto \partial F$ yields an
isometric inclusion of Hilbert spaces
\[ {\sf H}[K]\;\hookrightarrow\; L^{2}[K] . \]

\vspace{4mm}

\begin{center}
\small{\bf 6.3 Hardy Grading for $K=\Q$ (\cite{GV}, \S 8.1)}
\end{center}

\vspace{2mm} 

\noindent 
 
Consider the case $K=\Q$.  
Denote 
by $\HP^{-}=\R\times i(-\infty ,0)$ the hyperbolic lower half-plane. Then every element $f=\sum a_{q}\xi^{q}\in L^{2}[\Q]$
determines a triple $(F_{-},F_{0}, F_{+})$
\[ F_{+}(\tau )\;\; =\;\; \sum_{q>0} a_{q}\eta^{q},\quad
\quad  F_{0}\;\;=\;\;  a_{0},\quad\quad 
F_{-}(\tau)\;\; =\;\; \sum_{q<0} a_{q}\bar{\eta}^{q}.
\]
The functions $F_{+}(\tau )$ and $F_{-}(\tau )$ are viewed as elements
of the Hardy spaces ${\sf H}_{+}[ \Q]= {\sf H}[\Q ]$ and 
${\sf H}_{-}[\Q ]$ = Hardy space of anti-holomorphic functions
on $\HSO_{\Q}$.  We thus obtain a graded Hilbert space
\[ {\sf H}_{\bullet}[\Q]= 
{\sf H}_{-}[\Q]\oplus\C \oplus{\sf H}_{+}[\Q]\cong L^{2}[\Q].
\]

\vspace{4mm}

\begin{center}
\small{\bf 6.4 $\theta$-Holomorficity  (\cite{GV}, \S 8.1)}
\end{center}

\vspace{2mm} 

\noindent To build from ${\sf H}[K]$ a holomorphic extension
of $K$, 
we must 
interpret all elements of $ L^{2}[K] $ 
as boundaries of holomorphic functions: that is,
one needs to introduce a generalized notion of ``anti-holomorphicity".  

Let $K/\Q$ be of degree $d$ and denote 
$\Theta_{K}=\{ -, +\}^{d}\cong (\Z/2\Z )^{d}$.
For each ${\boldsymbol \theta}\in
\Theta_{K}$,  define
\[ \HP_{K}^{{\boldsymbol \theta}}\;\;=\;\;\big\{ {\boldsymbol \tau}=(x_{\nu}+it_{\nu})\in \C_{K}
\;\big|\;\;
\big(\text{sign}(t_{\nu})) = {\boldsymbol \theta}
\big\} .\] 
Define ${\boldsymbol \theta}$-conjugation
${\sf c}_{{\boldsymbol \theta}}:\C_{K}\rightarrow\C_{K}$, where the coordinates
of  ${\sf c}_{{\boldsymbol \theta}}({\boldsymbol \tau})={\boldsymbol \tau}'$
satisfy
$x'_{\nu} +it'_{\nu} =x_{\nu} + \theta_{\nu}it_{\nu}$.
The conjugation maps are not holomorphic in ${\boldsymbol \tau}$ and satisfy
for all ${\boldsymbol \theta}, {\boldsymbol \theta}_{1},{\boldsymbol \theta}_{2}\in\Theta_{K}$:
\[  {\sf c}_{{\boldsymbol \theta}}(\HP_{K}) = \HP_{K}^{{\boldsymbol \theta}} \quad
\text{and}\quad
{\sf c}_{{\boldsymbol \theta}_{1}}\circ{\sf c}_{{\boldsymbol \theta}_{2}}
={\sf c}_{{\boldsymbol \theta}_{1}{\boldsymbol \theta}_{2}}.\]
A  \begin{color}{RoyalBlue} {\bf {\em $\theta$-holomorphic function}}\end{color} is one of the form $F\circ c_{\theta}$, where $F: \C_{K}\rightarrow \C$ is holomorphic.

\vspace{4mm}

\begin{center}
\small{\bf 6.5 Hardy Grading for $K\not=\Q$ (\cite{GV}, \S 8.1)}
\end{center}

\vspace{2mm} 

\noindent Denote 
by $K^{{\boldsymbol \theta}}$ those elements $\alpha\in K$ whose coordinates with
respect to the embedding $K\hookrightarrow K_{\infty}$ satisfy 
$\text{sign}(\alpha )= (\text{sign}(\alpha_{\nu})) = {\boldsymbol \theta}.$
Note that  
\[ K^{{\boldsymbol \theta}_{1}}\cdot K^{{\boldsymbol \theta}_{2}}\subset 
K^{{\boldsymbol \theta}_{1}{\boldsymbol \theta}_{2}}.\]   

Then every element $f=\sum a_{\alpha}{\boldsymbol \xi}^{\alpha}\in  L^{2}[K] $ determines
a $(2^{d}+1)$-tuple $(F_{{\boldsymbol \theta}}; F_{0})$, where for each 
${\boldsymbol \theta}\in\Theta_{K}$, $F_{{\boldsymbol \theta}}:\HSO_{K}\rightarrow\C$ is defined
as the unique extension to $\HSO_{K}$ of the following $\theta$-holomorphic function on $\HP_{K}$:
\[ F_{{\boldsymbol \theta}}({\boldsymbol \tau})\;\; =\;\; \sum_{\alpha\in K^{{\boldsymbol \theta}}} 
a_{\alpha} {\sf c}_{{\boldsymbol \theta}}({\boldsymbol \eta})^{\alpha} \]
where ${\sf c}_{{\boldsymbol \theta}}({\boldsymbol \eta})^{\alpha} = 
 \exp \big(2\pi i \, {\rm Tr}( \alpha \cdot {\sf c}_{{\boldsymbol \theta}}({\boldsymbol \tau})) \big)$.
The term 
$F_{0}$ is the constant function $a_{0}$. 
The Hardy space of $\theta$-holomorphic functions is 
denoted ${\sf H}_{{\boldsymbol \theta}}[K]$.
The space of $(2^{d}+1)$-tuples
is viewed as a graded Hilbert space:
\[ {\sf H}_{\bullet}[K] = 
\Big(\bigoplus_{{\boldsymbol \theta}}{\sf H}_{{\boldsymbol \theta}}[K]\Big)\oplus\C 
\cong  L^{2}[K] ,
\]
whose inner-product is the direct sum of the inner-products
on each of the summands.
We write ${\sf H}[K]$ for the summand of
${\sf H}_{\bullet}[K]$ corresponding to ${\boldsymbol \theta} =(+,\dots ,+)$
{\em i.e.}\ the Hardy space of functions holomorphic on $\HSO_{K}$
in the ordinary sense.

The projectivization of the $\theta$-Hardy spaces $\PR{\sf H}_{{\boldsymbol \theta}}[K]$ yields an arrangement of projective subspaces of $\PR {\sf H}_{\bullet}[K]$.  We
write ${\sf N}_{\theta}[K]=  {\sf N}[K]\cap \PR{\sf H}_{{\boldsymbol \theta}}[K]$
and denote by ${\sf N}_{\bullet}[K]$ the nonlinear number field ${\sf N}[K]$
equipped with the arrangement of ``Hardy homogeneous" subspaces ${\sf N}_{\theta}[K]$.

\vspace{4mm}

\begin{center}
\small{\bf 6.6 Dirichlet Homogeneity (\cite{GV}, \S 8.1)}
\end{center}

\vspace{2mm} 

\noindent The Cauchy and Dirichlet products are defined on ${\sf H}_{\bullet}[K]$
via boundary extensions: that is,  $F\otimes G$ is defined
to be the unique element of ${\sf H}_{\bullet}[K]$ whose
boundary is $\partial F\otimes\partial G$, provided the latter is defined.  
The Cauchy product does not 
respect the grading, but the Dirichlet product 
has the following graded decomposition law:
\[ \big(F\otimes G\big)_{{\boldsymbol \theta}} \;\;=\;\;
\sum_{{\boldsymbol \theta}={\boldsymbol \theta}_{1}{\boldsymbol \theta}_{2}}F_{{\boldsymbol \theta}_{1}}\otimes G_{{\boldsymbol \theta}_{2}} 
\quad\quad\quad\quad
\big( F\otimes G\big)_{0} \;\;=\;\;
 F(1) G(1)\;\; -\;\; F_{0} G_{0}.
\]

\vspace{4mm}

\begin{center}
\small{\bf 6.7 Infinite Degree Extensions (\cite{GV}, \S 8)}
\end{center}

\vspace{2mm} 

\noindent  Let $\mathcal{K}/\Q$ be of infinite degree and let  $\HSO_{\mathcal{K}}$ 
be the hyperbolized adele class group associated to $\SO_{\mathcal{K}}$.
A continuous function $F:\HSO_{\mathcal{K}}\rightarrow\C$
is holomorphic if its restriction $F|_{\HP_{\mathcal{K}}}$
is holomorphic i.e. if $F$ is holomorphic separately in each factor of the
polydisk decomposition $\HP_{\mathcal{K}}= \prod \HP_{\nu}$.

Define as before 
$\SO_{\mathcal{K}} ({\bf t})$ = the subset of $\HSO_{\mathcal{K}} $
having extended coordinate ${\bf t}\in (0,\infty )^{\infty}$.  The
proto compactification of $\SO_{\mathcal{K}} ({\bf t})$ is a lamination
$\hat{\SO}_{\mathcal{K}} ({\bf t})$
homeomorphic to the proto adele classe group $\hat{\SO}_{\mathcal{K}}$.  
If we let
\[ \hat{\HSO}_{\mathcal{K}}\;\;=\;\;\hat{\SO}_{\mathcal{K}}\times
(0,\infty )^{\infty}\;\;=\;\;\bigcup\; \hat{\SO}_{\mathcal{K}} ({\bf t}),\]
the Hardy
space ${\sf H}[\mathcal{K}]$ is defined as the space
of holomorphic functions $F:\HSO_{\mathcal{K}}\rightarrow\C$
having a continuous extension 
$\hat{F}:\hat{\HSO}_{\mathcal{K}} \rightarrow\C$
and for which the norm
\[ \| \hat{F}\|_{\bf t}^{2}\;\;=\;\;
\int_{\hat{\SO}_{\mathcal{K}}({\bf t}) } |\hat{F}|^{2}\,d\mu \] 
is uniformly bounded in ${\bf t}$ (where $d\mu$
is induced from the Haar measure $\mu$ on 
$\hat{\SO}_{\mathcal{K}}\approx\hat{\SO}_{\mathcal{K}}({\bf t})$).

As in the finite-dimensional case, ${\sf H}[\mathcal{K}]$
is a Hilbert space with respect to the supremum of the integration
pairings on each $\hat{\SO}_{\mathcal{K}} ({\bf t})$.
With this definition in place, the discussion in \S\S 6.4, 6.5 proceeds
without modification.  Note here that the sign group $\Theta_{\mathcal{K}}$
is a countable dense subgroup of the Cantor group $\{-,+\}^{\infty}$. 

 \pagebreak

\vspace{4mm}

\begin{center}
\small{\bf 6.8 Galois Sign Representation}
\end{center}

\vspace{2mm} 

\noindent If $L/K$ is Galois, then $\sigma\in {\rm Gal}(L/K)$ permutes
the grading of ${\sf H}_{\bullet}[L]$ and induces an automorphism of the sign group $\Theta_{K}$
that acts trivially on the subgroup $i_{L/K}(\Theta_{K})$ of  signs corresponding to elements in
$i_{L/K}(K_{\infty})$.  In other words, there is 
a canonical representation
\[  {\rm Gal}(L/K)\rightarrow {\rm Aut}(\Theta_{L} /i_{L/K}(\Theta_{K})). \]
This is true for the infinite degree field extension $\bar{\Q}/\Q$, in which case
we obtain a representation of ${\rm Gal}(\bar{\Q}/\Q)$ in the group of automorphisms of a dense subgroup of a Cantor group.

\vspace{4mm}

\begin{center}
\small{\bf 6.9 Abstract Nonlinear Number Field}
\end{center}

\vspace{2mm} 

\noindent  As we shall see, it is useful to consider the notion of nonlinear number field more broadly, as a mantle under which
several closely related spaces lie.  Thus by an \begin{color}{RoyalBlue} {\bf {\em abstract nonlinear number field}}\end{color} we will mean a graded topological partial double semigroup 
${\sf S}_{\bullet}=({\sf S}_{\bullet}, \oplus, \otimes )$ for which 
\begin{itemize}
\item[-] There exists a (possibly infinite degree) algebraic number field $K/\Q$
and a graded continuous monomorphism ${\sf N}_{0}[K] \hookrightarrow {\sf S}_{\bullet}$
with dense image.  Here, ${\sf N}_{0}[K]$ inherits the grading of ${\sf N}_{\bullet}[K]$.
\item[-] The element $1_{\oplus}\in {\sf S}_{\bullet}$ acts as a universal annihilator for the operation
$\otimes$: for all $f\in {\sf Dom}_{\otimes}(1_{\oplus})$, $f\otimes 1_{\oplus}=1_{\oplus}$.
\end{itemize}

All of the spaces ${\sf N}_{0}[K]$, ${\sf N}_{\bullet}[K]$, ${\sf W}_{\bullet}[K]$
and $\bar{\sf N}[K]$ are abstract nonlinear number fields.
 
 \section{Galois Theory}
 
We continue here to restrict attention
to totally real fields $K/\Q$ to simplify the exposition, noting
that the general case is discussed in \cite{GV}, \S 10.

\vspace{4mm}

\begin{center}
\small{\bf 7.1 Nonlinear Automorphisms (\cite{GV}, \S 10)}
\end{center}

\vspace{2mm} 

\noindent Let $K/\Q$ and 
equip ${\sf N}_{\bullet}[K]$ with the Fubini-Study metric
associated to the inner-product on ${\sf H}_{\bullet}[K]$.
A  \begin{color}{RoyalBlue} {\bf {\em nonlinear automorphism}}\end{color} 
$\Upsilon:{\sf N}_{\bullet}[K]\rightarrow
{\sf N}_{\bullet}[K]$
is the restriction of a Fubini-Study isometry of $\PR{\sf H}_{\bullet}[K]$
respecting the grading and the operations $\oplus$, $\otimes$
whenever they are defined:
that is, 
\[ \Upsilon ([F]\oplus [G])\;\; =\;\; \Upsilon ([F])\oplus \Upsilon ([G]),
\quad\quad\quad\quad
\Upsilon ([F]\otimes [G])\;\; =\;\; \Upsilon ([F])\otimes \Upsilon ([G])\] and
for some permutation $\iota$ of $\Theta_{K}$, 
$\Upsilon (\PR{\sf H}_{\boldsymbol \theta}[K]) = \PR{\sf H}_{\iota ({\boldsymbol \theta} )}[K]$ for all ${\boldsymbol \theta} \in \Theta_{K}$.
An identical definition is made in the case of an infinite degree algebraic extension.

\vspace{4mm}

\begin{center}
\small{\bf 7.2 Nonlinear Galois Groups (\cite{GV}, \S 10)}
\end{center}

\vspace{2mm} 

\noindent If $L/K$ is a finite field extension denote by 
\[{\sf Gal}\big({\sf N}_{\bullet}[L]/{\sf N}_{\bullet}[K]\big) \]
the group of automorphisms of 
${\sf N}_{\bullet}[K]$ acting trivially on
the sub nonlinear field ${\sf N}_{\bullet}[K]$, and
by 
\[{\sf Gal}\big({\sf N}_{\bullet}[K]\big/K\big)\] the group of 
automorphisms of ${\sf N}_{\bullet}[K]$
acting trivially on $K$.

\begin{theo}\label{GalIsOne}  
${\sf Gal}\big({\sf N}_{\bullet}[K]\big/K\big)\cong\{ 1\}$.
\end{theo}

This is proved as follows: 
if $\sigma\in {\sf Gal}\big({\sf N}_{\bullet}[K]\big/K\big)$ then
by Wigner's Theorem \cite{Wi}, $\sigma$ is the projectivization
of an (anti) unitary linear map 
$\tilde{\sigma}: {\sf H}_{\bullet}[K]\longrightarrow
{\sf H}_{\bullet}[K]$.
Since $\sigma$ fixes $K$, there exist multipliers $\lambda_{\alpha}\in U(1)$
with $\tilde{\sigma}({\boldsymbol \eta}^{\alpha}) = \lambda_{\alpha}\cdot {\boldsymbol \eta}^{\alpha}$.
But $\sigma$ respects the Cauchy and Dirichlet products, wherein we must have
that $\lambda$ is simultaneously an additive and multiplicative character: 
\[ \lambda_{\alpha_{1}+\alpha_{2}}=\lambda_{\alpha_{1}}\lambda_{\alpha_{2}}=\lambda_{\alpha_{1}\alpha_{2}},\]
possible only for $\lambda$ trivial.

In \cite{GV} the following fundamental theorem is proved:

\begin{theo}\label{GalIsGal} Let $L/K$ be
a Galois extension. Then 
\[ {\sf Gal}\big({\sf N}_{\bullet}[L]/
{\sf N}_{\bullet}[K]\big) 
\;\;\cong\;\; {\sf Gal}(L/K ).\] 
\end{theo}
The results in this section are also true without modification for infinite
degree algebraic extensions.

\vspace{4mm}

\begin{center}
\small{\bf 7.3 Cauchy and Dirichlet Flows  (\cite{GV}, \S 10)}
\end{center}

\vspace{2mm} 

\noindent In what follows it will be convenient to work with the abstract nonlinear number field
$\bar{\sf N}_{\bullet}[K]=\PR {\sf H}_{\bullet}[K]$.
Assume now that $K/\Q$ is of finite degree and
consider each of the operations
$\oplus$ and $\otimes$ separately.  
Let ${\sf Gal}_{\oplus}\big(\bar{\sf N}_{\bullet}[K]/K\big)$
denote the group of isometries of $\bar{\sf N}_{\bullet}[K]$ that
fix $K$, permute the grading and are homomorphic with respect
to $\oplus$ only.  We similarly define
${\sf Gal}_{\otimes}\big(\bar{\sf N}_{\bullet}[K]/K\big)$.

Denote by ${\sf U}\big({\sf H}_{\bullet}[K]\big)$
the group of unitary or anti-unitary operators of ${\sf H}_{\bullet}[K]$.
The action of ${\bf r}\in K_{\infty}$ by 
translation in $\HP_{K}$, 
${\boldsymbol \tau}\mapsto {\boldsymbol \tau}+{\bf r}$, induces an action
on ${\sf H}_{\bullet}[K]$ by
\[  \Phi_{\bf r}(F)\;\; =\;\;   
\sum a_{\alpha}\exp (2\pi i{\rm Tr}( \alpha {\bf r}) )\boldsymbol{\eta}^{\alpha} \]
for $F=\sum_{\alpha}a_{\alpha}\boldsymbol{\eta}^{\alpha}$, yielding a faithful representation $\Phi :K_{\infty}\longrightarrow 
{\sf U}\big({\sf H}_{\bullet}[K]\big)$.
In particular, for each ${\bf r}$,  $\Phi_{\bf r}$ is an isometry that acts trivially on the grading i.e.\  $\Phi_{\bf r}({\sf H}_{\theta})={\sf H}_{\theta}$
for $\theta\in\Theta_{K}$.  It is straightforward to check that $\Phi$
induces a monomorphism 
\[ \Phi : K_{\infty}\;\hookrightarrow \;
{\sf Gal}_{\oplus}\big(\bar{\sf N}_{\bullet}[K]/K\big)
. \]
We call $\Phi$ the  \begin{color}{RoyalBlue} {\bf {\em Cauchy flow}}\end{color}.

\begin{note} The representation $\Phi$ {\it does not} preserve the subspace $\widetilde{\sf Z}[K]$ of trace zero elements, hence does not preserve ${\sf N}[K]$.   It may viewed as a flow {\it with singularities} on ${\sf N}[K]$.
\end{note}

We also define a  $d$-dimensional flow on
${\sf H}_{\bullet}[K]$ respecting $\otimes$ as follows.
For a vector
${\bf x}\in K_{\infty}$, write $\log |{\bf x}|=(\log |x_{\nu}|)$.
Then for
$F=\sum_{\alpha}a_{\alpha}{\boldsymbol \eta}^{\alpha}$ we define
\[  \Psi_{\bf r} (F)\;\; =\;\; 
\sum_{\alpha\in K}a_{\alpha}
\exp \big( 2\pi i{\rm Tr}({\bf r} \log | \alpha | )\big) 
{\boldsymbol \eta}^{\alpha}\;\; =\;\; \sum_{\alpha\in K}a_{\alpha}
 | \alpha |^{2\pi i {\bf r}}  
{\boldsymbol \eta}^{\alpha} .\]
This defines a faithful representation
$\Psi :K_{\infty}\longrightarrow   {\sf U}\big({\sf H}_{\bullet}[K]\big)$ inducing
a monomorphism
\[ \Psi :K_{\infty}\;\hookrightarrow \;
{\sf Gal}_{\otimes}\big(\bar{\sf N}_{\bullet}[K]/K\big)
. \]
We call $\Psi$ the  \begin{color}{RoyalBlue} {\bf {\em Dirichlet flow}}\end{color}.

\vspace{4mm}

\begin{center}
\small{\bf 7.4 Parallels Between the Dirichlet Flow and Berry's Hypothetical Riemann Flow}
\end{center}

Let $K=\Q$.  Notice then that
the periodic orbits of the Dirichlet flow have period $1/\log |q|$ for any $q\in\Q$.   There
is also a time-reversal symmetry: a conjugation $T\circ \Psi_{r}\circ T^{-1}=\Psi_{-r}$ defined by the Dirichlet automorphism
\[ f=\sum a_{q}\xi^{q}\longmapsto T(f)= \sum a_{q}\xi^{q^{-1}} . \] 

The subspace $\bar{\sf N}(\N)\subset \bar{\sf N}(\Q )$ of classes corresponding to $\N$-indexed
Fourier series is invariant by the Dirichlet
flow and the periodic orbits have length $1/\log n$, $n\in\N$.  However
the operator $T$ defined above does not preserve $\bar{\sf N}(\N)$, 
and we conjecture
that the system $(\bar{\sf N}(\N), \Psi_{r})$ does not have time-reversal symmetry.  This system 
thus bears some similarity with the hypothetical Riemann flow of Berry \cite{Ber}, \cite{Con}.

\vspace{2mm} 

\vspace{4mm}

\begin{center}
\small{\bf 7.5 Idele Class Group as  a Galois Group  (\cite{GV}, \S 10)}
\end{center}

\vspace{2mm}

Recall that by abelian class field theory \cite{Ne}, the  \begin{color}{RoyalBlue} {\bf {\em idele class group}}\end{color} ${\sf C}_{\Q}=\A_{\Q}^{\times}/\Q^{\times}$ of $\Q$ may be
identified with $\R^{\times}_{+}\times {\rm Gal}(\bar{\Q}^{\rm ab}/\Q)$.  
As a corollary to the discussion in \S\S 7.2, 7.3, there are monomorphisms
\[ {\sf C}_{\Q}\;\hookrightarrow\; 
{\sf Gal}_{\oplus}\big(\bar{\sf N}_{\bullet}[\Q^{\rm ab}]/\Q\big) ,
\quad\quad\quad\quad
{\sf C}_{\Q }\;\hookrightarrow\; 
{\sf Gal}_{\otimes}\big(\bar{\sf N}_{\bullet}[\Q^{\rm ab}]/\Q\big) .
\]
The extension $\Q^{\rm ab}/\Q$ is complex, so the grading 
of ${\sf N}_{\bullet}[\Q^{\rm ab}]$ is the
complex grading of \cite{GV} \S\S 6, 8,  which we have not discussed in these notes.

 \section{Nonlinear Number Fields, L-Functions and Modular Forms}

 \vspace{4mm}

\begin{center}
\small{\bf 8.1 Riemann Zeta-Function}
\end{center}

\vspace{2mm} 

The translated Riemann zeta function $\zeta (s+1)$ is
holomorphic in the right half plane $\{s|\; {\rm Re} (s)>0\}$.  
The substitution $n^{-s}\mapsto \eta^{n}$ 
transforms\footnote{The substitution $ \eta^{n} \mapsto n^{-s}$ is the Mellin transform
modulo Gamma factors:
\[  n^{-s} = (2\pi)^{s} \Gamma (s)^{-1} {\rm Mellin}(\eta^{n}(iy))   \]
where $\eta (iy)=\exp (-2\pi y)$.} $\zeta (s+1)$ to
\[ {\rm Li}_{1}(\eta ) = \sum_{n\geq 1} (1/n)\eta^{n}  = -\log (1- \eta ).  \]
 In general the translates $\zeta (s+ s_{0} )$ with 
${\rm Re}(s_{0} )\geq 1$  yield the polylogarithms ${\rm Li}_{s_{0}} (\eta )\in{\sf H}[\N]\subset {\sf H}[\Q]$.  

We remark that the translated zeta functions
$\zeta (s+s_{0} )$ enjoy all of the usual properties present in $\zeta (s)$: meromorphic continuation, Euler product, functional equation and boundedness in vertical strips.  Of course, the Riemann Hypothesis is true of $\zeta (s)$ if and only if its true for $\zeta (s+s_{0}  )$, in the sense that all non-trivial zeros lie on a common vertical line.

Notice that ${\rm Li}_{1}(\eta )$, defined for $\eta\in\D^{\ast}$, has a multi-valued analytic continuation to $\eta\in\C-1$, with no zeros in $\D^{\ast}$, and only one at $\eta =0$ after continuation.  Moreover, the pole at $s=0$ transforms to a {\it logarithmic singularity} at $\eta = 1$. 
This example indicates that the ``inverse Mellin transform" 
$n^{-s}\mapsto \eta^{n}$
does not transport the analytic continuation and the zeros of $\zeta (s+1)$  to
their obvious counterparts in the $\eta$-plane.

\vspace{4mm}

\begin{center}
\small{\bf 8.2  L-Functions}
\end{center}

\vspace{2mm}

 Consider an L-function written written in the form of a Dirichlet series
\[  L(s) = \sum a_{n}n^{-s}.\]
We assume that $L(s)$ is holomorphic in the right half plane ${\rm Re} (s)>0$, and if it is not, we shift by a suitable constant so that it is.    Assume also that the sequence of coefficients $\{ a_{n}\}$ belongs to $l^{2}$.  If we define $f_{L}$
via the substitution $n^{-s}\mapsto \eta^{n}$,  then $f_{L}(\eta)\in {\sf H}[\N]$.
When $L_{1}\cdot L_{2}$ has $l^{2}$ coefficients then 
 \[  f_{L_{1}\cdot L_{2}} =  f_{L_{1}}\otimes f_{L_{2}},  \]
so that the correspondence $f\mapsto f_{L}\in {\sf N}[\N]$ takes
the algebra of such L-functions to the Dirichlet algebra of ${\sf N}[\N]$.


 \vspace{4mm}

\begin{center}
\small{\bf 8.3 Modular Forms and the Hecke Algebra}
\end{center}

\vspace{2mm}

 \noindent For simplicity we confine ourselves to the case of modular forms
 for $\Gamma (1)={\rm SL}_{2}(\Z )$. Let $S_{k}$ be the space of cusp forms of weight $k$.  As in the case of the zeta function, we will renormalize
 the Fourier coefficients of $f\in S_{k}$ 
 in order to obtain an element of ${\sf H}[\N]$:
 that is, we replace $a_{n}$ by $\lambda_{n}=a_{n}/n^{k/2}$.
This corresponds to translating\footnote{The translation by $(k-1)/2$ is often employed to shift the axis of symmetry of the functional equation to ${\rm Re}(s)=1/2$.} the associated L-function 
 by $k/2$.

\begin{prop}   $S_{k}\subset
{\sf H}[\N ]$.
\end{prop}

\begin{proof}  The Fourier expression for $f$
gives a well-defined holomorphic function on $\mathfrak{T}_{\Q}\cong \Delta^{\ast}$.
We must show that the sequence of Fourier coefficients $\{\lambda_{n}\}$
belong to $l^{2}$.  It suffices to show this for a normalized Hecke eigenform.  Then by Deligne's Theorem (Ramanujan-Petersson conjecture) \cite{De}, there is a constant $C>0$ such that
\[  |\lambda_{n}|^{2} = n^{-k} |a_{n}|^{2} \leq  Cn^{-k}n^{k-1+\epsilon} \leq C n^{-1+\epsilon}. \]
Thus $f\in {\sf H}[\N ]\subset  {\sf H}[\Q ]$.
\end{proof}

Denote by
 \[  {\rm pr}_{\Q/\Z}: {\sf H}[\Q]\longrightarrow {\sf H}[\Z]\]
 the orthogonal projection operator.
 Consider for each prime $p$ the Puiseux polynomial
\[   t_{p}(\eta ) = \eta^{p} + p^{k-1}\eta^{1/p}\in  {\sf H}[\Q].  \]

\begin{prop}  The operator 
\[  T_{p }=  {\rm pr}_{\Q/\Z}\circ M_{\otimes}^{t_{p}}  \]
coincides with the usual Hecke operator when restricted to $S_{k}$.
\end{prop}

\begin{proof}  If $f\in S_{k}$ then 
\[   T_{p}(f) =  {\rm pr}_{\Q/\Z}\bigg(\sum_{n>0} a_{n}\big(\eta^{np} +   
p^{k-1}\eta^{n/p}\big) \bigg) = 
\sum_{m>0}\big(\bar{a}_{m/p} + p^{k-1}\bar{a}_{mp}\big)\eta^{m}\]
where $\bar{a}_{x}=a_{x} $ if $x\in\N$ and $0$ otherwise.  But this is precisely
the Fourier expansion of the image of $f$ by the usual Hecke operator, see Lemma 5.1 of \cite{Kow}.
\end{proof}

Thus the Hecke algebra is a subalgebra of the Dirichlet algebra of ${\sf H}[\Q]$
with an added projection.  
  
  \section{Aperiodic Unit Class Group}
  
  The ideas contained in this section are something of an afterthought, 
  and we do not claim to present them in definitive form.  We thank Jim Cogdell for taking the time
  to clarify a few points about Artin L-functions.
  
 \vspace{4mm}

\begin{center}
\small{\bf 9.1 Aperiodic Unit Class Group}
\end{center}

\vspace{2mm} 

\noindent Recall that $W[\Q ]$ is the Wiener field algebra of $\Q$ (see \S 5.5).  Denote by $D_{+}[\Q]$ the group of Dirichlet units in ${\sf H}[\Q]\cap W[\Q]$:
the group of $f\in {\sf H}[\Q]\cap W[\Q]$ such that there exists $g\in {\sf H}[\Q]\cap W[\Q]$ with $f\otimes g=1_{\otimes}$.
Let ${\sf D}_{+}[\Q]$ be the projectivization of $D_{+}[\Q]$.  Note that $\Q_{+}^{\times}\subset {\sf D}_{+}[\Q]$.   We call
the quotient group
\[  \mathfrak{D}[\Q] :=  {\sf D}_{+}[\Q] /\Q_{+}^{\times}\]
the \begin{color}{RoyalBlue} {\bf {\em unit class group}}\end{color}.  

Let ${\sf D}[\N]$ = the subgroup of ${\sf D}_{+}(\Q)$ of units of the form $f(\eta ) =   \sum_{n>0} a_{n} \eta^{n}$.  Equivalently \cite{Ap},  ${\sf D}[\N]$ consists of those $\N$-indexed
power series for which $a_{1}\not=0$.  For example ${\rm Li}_{2}(\eta )$ (the nonlinear integer associated to $\zeta (s+2)$, see \S 8.1) defines an element of ${\sf D}[\N]$ since it has
inverse $g= \sum_{n>0} b_{n} \eta^{n} $
where $b_{n}=\mu(n)/n^{2}$ and $\mu (n)$ is the M\"{o}bius function.  Note that
 ${\sf D}[\N]\cap \Q^{\times}_{+}=1_{\otimes}$ so that
${\sf D}[\N]$ embeds in $ \mathfrak{D}[\Q] $; we denote its image 
by $ \mathfrak{D}[\N] $.

Let $\mathfrak{D}^{\star}[\N] $ denote the subgroup of projective classes of series containing
a representative whose coefficients $a_{n}$ are completely multiplicative.  Consider the subgroup
$\mathfrak{P}^{\star}[\N ]\subset \mathfrak{D}^{\star}[\N] $
generated by elements of the form $f_{p}=\sum a_{p^{k}}\eta^{p^{k}}$.  Elements of 
$\mathfrak{P}^{\star}[\N ]$ have power series expansions involving powers $\eta^{n}$ for which $n$ is divisible by only a finite set of primes.  In particular, $\mathfrak{P}^{\star}[\N ]$ consists of all elements of 
$\mathfrak{D}^{\star}[\N] $ lying on a periodic orbit of the Dirichlet flow.
The quotient 
\[  \mathfrak{D}_{\rm ap}^{\star}[\N] =\mathfrak{D}^{\star}[\N] /\mathfrak{P}^{\star}[\N ] \]
is called the \begin{color}{RoyalBlue} {\bf {\em aperiodic unit class group}}\end{color}.

\vspace{4mm}

\begin{center}
\small{\bf 9.2 Action by the Dual Idele Class Group and the Galois Character Group}
\end{center}

\vspace{2mm} 

\noindent 

The Dirichlet flow $\Psi$ stabilizes $\Q^{\times}_{+}$, acts by automorphisms of ${\sf D}_{+}[\Q]$ and so induces an action of 
$ \mathfrak{D}[\Q]$ which fixes set-wise the subgroup $ \mathfrak{D}[\N] $.  Moreover it stabilizes the subgroups $\mathfrak{P}^{\star}[\N ]\subset \mathfrak{D}^{\star}[\N] $ and so descends to a 1-parameter group of automorphisms of $\mathfrak{D}^{\star}_{\rm ap}[\N ]$.

Consider the Galois character group ${\rm Char}({\rm Gal}(\bar{\Q}/\Q))\cong{\rm Char}(\hat{\Z}^{\times})$.  For each element $\bar{\chi}\in {\rm Char}(\hat{\Z}^{\times})$ we associate a (primitive) Dirichlet character $\chi :\Z\rightarrow {\sf U}(1)$ induced by $\bar{\chi}$, which we recall is completely multiplicative.   Denote by $\chi_{n}$ the value of $\chi$ at $n$. 
We define an endomorphism $R_{\chi}$ of $ W[\N]$ by
\begin{equation}\label{DirCharAct}
  R_{\chi}(f) = \sum_{n>0} \chi_{n}a_{n}\eta^{n}. 
  \end{equation}
Observe first that $R_{\chi}$ yields a self-map of $W[\N]$ since $\| R_{\chi}( f) \| \leq
\sum |a_{n}|=\| f\| $ so that $ R_{\chi}( f)\in W[\N]$.
Moreover $R_{\chi}$ is a Dirichlet endomorphism:
\begin{eqnarray*}  R_{\chi}(f\otimes g) & = &   \sum_{n>0} \chi_{n}\bigg(\sum_{n_{1}n_{2}=n}
a_{n_{1}} b_{n_{2}}\bigg)\eta^{n}  \\
 & = &  
 \sum_{n>0} \bigg(\sum_{n_{1}n_{2}=n}
\chi_{n_{1}}a_{n_{1}} \chi_{n_{2}}b_{n_{2}}\bigg)\eta^{n} \\
& = &  R_{\chi}( f)\otimes  R_{\chi}( g).
\end{eqnarray*}
Note that $\chi$ has nontrivial kernel: for example, if $\chi$
has conductor $N$
and $f=\eta + \eta^{N}$ then $R_{\chi}( f) = \eta = 1_{\otimes}$
since $\chi_{N}=0$.
The endomorphism $R_{\chi}$ descends to one of $\mathfrak{D}(\N )$ preserving
the subgroups $\mathfrak{P}^{\star}[\N ]\subset \mathfrak{D}^{\star}[\N] $ and induces an {\it automorphism} of $\mathfrak{D}^{\star}_{\rm ap}[\N ]$ (since ${\rm Ker}(R_{\chi})\subset
\mathfrak{P}^{\star}[\N ]$).

\begin{theo} The map $\bar{\chi}\mapsto R_{\chi}$ induces a well-defined monomorphism of 
${\rm Char}(\hat{\Z}^{\times})$
into ${\rm Aut}(\mathfrak{D}^{\star}_{\rm ap}[\N ])$ taking the product of characters to composition
of automorphisms.
\end{theo}

\begin{proof}  Denote the Dirichlet product\footnote{Usually called the {\it Dirichlet convolution}
 \cite{Ap}.} of arithmetic functions $\chi$, $\psi$  by $\chi\ast\psi$.  Let $R_{\chi\ast\psi}$ be defined as in (\ref{DirCharAct}) using
the arithmetic function $\chi\ast\psi$.   Then for any $f\in\mathfrak{D}^{\star}(\N )$
\begin{equation}\label{convisprod} R_{\chi\ast\psi}(f) =R_{\chi}(f)\otimes R_{\psi}(f) , \end{equation}
where the complete multiplicativity of $f$ is essential in the verification.
Now consider the arithmetic function $\zeta_{p}$ defined $\zeta_{p}(n)=1$ if $n=p^{k}$ for some
$k\geq 0$
and $0$ otherwise\footnote{This is just the arithmetic function associated to the local zeta function
$(1-p^{-s} )^{-1}$.}.  If $\chi'$ is a Dirichlet character induced by $\bar{\chi}$
which is not equal to the primitive character $\chi$, then there are primes $p_{i}$ (possibly occurring with multiplicity) such that (see \S 3 of \cite{Kow1})
\begin{equation}\label{nonprim}
 \chi = \chi'\ast  \zeta_{p_{1}} \ast \cdots \ast  \zeta_{p_{m}}.  
 \end{equation}
It follows by (\ref{convisprod}) that in $ \mathfrak{D}^{\star}_{\rm ap}[\N]$
\begin{equation}\label{welldefdir}
 R_{\chi}(f) = R_{\chi'}(f)  
 \end{equation}
since $R_{\zeta_{p}}(f)\in  \mathfrak{P}^{\star}[\N ]$ for all $p$.  In particular, the map $\bar{\chi}\mapsto R_{\chi}$
is independent of the choice of Dirichlet character $\chi$ induced by $\bar{\chi}$.
While the product $\chi\psi$ of primitive Dirichlet characters induced by $\bar{\chi}$, $\bar{\psi}$
need not be primitive, it is nevertheless a Dirichlet character induced by the product $\bar{\chi}\bar{\psi}$, thus it follows by (\ref{nonprim}) and (\ref{welldefdir}) that $R_{\chi\psi}=R_{\chi}\circ R_{\psi}$.
\end{proof}

Let ${\sf C}_{\Q}$ be the idele class group of $\Q$. Since $(\R ,+)\cong
\R^{\times}_{+}\cong {\rm Char}(\R^{\times}_{+} )$ and 
$ {\rm Char}({\sf C}_{\Q})\cong {\rm Char}(\hat{\Z}^{\times})\times {\rm Char}(\R^{\times}_{+})$,
we may define an action of $ {\rm Char}({\sf C}_{\Q})$ by automorphisms of 
$ \mathfrak{D}^{\star}_{\rm ap}[\N]$:
that is, an embedding $ {\rm Char}({\sf C}_{\Q})
\subset {\rm Aut}(  \mathfrak{D}^{\star}_{\rm ap}[\N])$.
Or by Pontryagin duality
\begin{equation}\label{cqcontend}
  {\sf C}_{\Q}\supset {\rm Char}( {\rm Aut}(\mathfrak{D}^{\star}_{\rm ap}[\N] )) . 
  \end{equation}
Question: Is $ {\sf C}_{\Q}={\rm Char}({\rm Aut}(  \mathfrak{D}^{\star}_{\rm ap}[\N])$?

\vspace{4mm}

\begin{center}
\small{\bf 9.3 Relatively Prime Dirichlet Product}
\end{center}

\vspace{2mm}

L-functions with Euler products have coefficients that are multiplicative but not necessarily completely multiplicative.  To find actions by Artin or automorphic L-functions, we must consider a variation of the Dirichlet product available for $\N$-indexed series in which the summation law is restricted to factorizations
$n_{1}n_{2}=n$ for which $(n_{1},n_{2})=1$.  Thus for
$f,g\in {\mathfrak D}(\N )$, we define the \begin{color}{RoyalBlue} {\bf {\em relatively prime Dirichlet product}}\end{color} by
\[f\check{\otimes} g =\sum_{n}\bigg(\check{\sum}_{
n_{1}n_{2}=n} a_{n_{1}}
b_{n_{2}} \bigg) \eta^{n}\]
where $\check{\sum}_{n_{1}n_{2}=n}$ means we
sum over pairs $n_{1},n_{2}$ which are relatively prime.

This product is certainly commutative. It is also associative:
\[ f\check{\otimes} (g\check{\otimes}h)=  \sum_{n}\bigg(\check{\sum}_{
n_{1}n_{2}=n} a_{n_{1}}
\bigg(\check{\sum}_{
n_{21}n_{22}=n_{2}} b_{n_{21}}
c_{n_{22}} \bigg) \bigg) \eta^{n} =
 \sum_{n}\bigg(\check{\sum}_{
n_{1}n_{2}n_{3}=n} a_{n_{1}}b_{n_{2}}
c_{n_{3}}\bigg)\eta^{n} 
 \]
where the last checked sum is over triples which are mutually relatively prime.
It follows then that $f\check{\otimes} (g\check{\otimes}h)=(f\check{\otimes} g)\check{\otimes}h$.
The formula for the inverse has the same recursive form as that
of the ordinary Dirichlet product \cite{Ap}, but where we replace $\sum$ by $\check{\sum}$ i.e. 
\[  f^{-1} = \frac{1}{f(1)}\eta +\sum_{n>1}b_{n}\eta^{n},\quad b_{n}= 
-\check{\sum}_{d|n,\; d<n} a_{n/d}b_{d}. \]
In particular, any $f\in {\mathfrak D}(\N)$ for which $a_{1}\not=0$ has a $\check{\otimes}$-inverse, so set-wise, the group of $\check{\otimes}$ units coincides with 
${\mathfrak D}(\N)$.
The  \begin{color}{RoyalBlue} {\bf {\em  relatively prime unit class group}}\end{color} $\check{\mathfrak{D}}(\N )$
is thus defined as the set ${\mathfrak D}(\N )$ with the operation $\check{\otimes}$.  

As in \S 9.1 we define $\check{\mathfrak{D}}^{\star}(\N )$ to be the subgroup of elements
which are {\it multiplicative} (not necessarily completely multiplicative), $\check{\mathfrak{P}}^{\star}[\N ]\subset \check{\mathfrak{D}}^{\star}(\N )$ the subgroup of Dirichlet periodic elements and let 
$\check{\mathfrak{D}}_{\rm ap}^{\star}(\N )=\check{\mathfrak{D}}^{\star}(\N )/\check{\mathfrak{P}}
^{\star}[\N ]$.

\vspace{4mm}

\begin{center}
\small{\bf 9.4 Action by Categories of Galois Representations, Automorphic Representations}
\end{center}

\vspace{2mm}

Let $\rho : {\rm Gal}(K/\Q )\longrightarrow {\rm GL}_{n}(\C )   $ be a linear Galois representation.  
Recall that for each prime $p$ the Euler factor of the associated Artin L-function $L(\rho, s)$ is 
\[ L_{p}(\rho, s) = {\det} (1-p^{-s}\rho (\sigma_{p} )  )^{-1}  \]
where $\sigma_{p}\in G_{p}/I_{p}$ is the Frobenius element, and where $G_{p}$ is the decomposition group and $I_{p}$ is the inertia group.  (If $p$ ramifies in $\rho$ then
the determinant is calculated in $(\C^{n})^{I_{p}}$ = subspace of fixed vectors of $\rho (I_{p})$.)
 Then the arithmetic function $\chi_{\rho}$ defined
 \[ L_{p}(\rho, s) = \sum_{n>0} \chi_{\rho ,n}n^{-s}  \]
 is multiplicative but not in general completely multiplicative.
 By definition of $\check{\otimes}$ it follows that 
 \begin{equation}\label{repaction}
  R_{\rho}(f) = \sum_{n>0} \chi_{\rho,n}a_{n}\eta^{n} 
  \end{equation}
defines an endomorphism of $\check{\mathfrak{D}}(\N )$ which descends to one
of $\check{\mathfrak{D}}_{\rm ap}^{\star}(\N )$.

Actually, one should consider $R_{\rho}$ as only being partially-defined, as it is not clear that the result
$R_{\rho}(f)$
is in the projective class of an element in $W[K]$, since we do not have a good understanding of the order of magnitude of the coefficients 
$\chi_{\rho,n}$.  Leaving aside this issue, one expects that (some variant of) (\ref{repaction}) will define a representation (i.e. a functor) of monoidal categories 
\[ R:{\rm Rep}({\rm Gal}(\bar{\Q}/{\Q}))\longrightarrow {\rm End}(\check{\mathfrak{D}}_{\rm ap}^{\star}(\N )).\]  Here, we regard ${\rm End}(\check{\mathfrak{D}}_{\rm ap}^{\star}(\N ))$  
as a category in which the morphisms are semi-conjugacies of endomorphisms.
Let us denote the ring operations of 
${\rm End}(\check{\mathfrak{D}}_{\rm ap}(\N ))$ by $\circ$ = composition (product) and $\boxplus$ (sum)
defined
\[ (S\boxplus S')(f) := S(f)\otimes S'(f) \quad (\text{Dirichlet product}).
\]
In this context, one has  $R_{\rho\oplus\sigma}  = R_{\rho} \boxplus R_{\sigma} $
since $\chi_{\rho\oplus\sigma}=\chi_{\rho}\ast\chi_{\sigma}$, see \S\S 2.1, 3.4 of \cite{deSh}.
Moreover in cases where the global Langlands correspondence has been verified we have \cite{Co}
that as endomorphisms of $\check{\mathfrak{D}}_{\rm ap}^{\star}(\N )$ (i.e. modulo $\mathfrak{P}^{\star}[\N ]$):
\begin{equation}\label{tensorproduct}   R_{\rho\otimes\sigma}  = R_{\rho} \circ R_{\sigma} .
\end{equation}
Similarly, a (variant of) (\ref{repaction}) for an automorphic representation
$\pi$ of ${\rm GL}_{n}(\A_{\Q} )$ should give rise to a corresponding representation of the category
of admissible cuspidal automorphic representations.

In the end, one anticipates being able to apply Tanaka duality to discover the analogue of (\ref{cqcontend}), though it is not clear to us yet what would play the role
of ${\rm Char}({\rm Aut}(\mathfrak{D}_{\rm ap}^{\star}(\N )))$, since there is no obvious analogue of the ``forgetful functor" from
${\rm End}(\check{\mathfrak{D}}_{\rm ap}^{\star}(\N ))$ to the category ${\rm Vect}_{\C}$ of finite dimensional vector spaces.


\begin{thebibliography}{00}

\bibitem [1]{Ap} Apostol, Tom M., {\it Introduction to Analytic Number Theory}, Undergraduate Texts in Mathematics, Springer-Verlag, New York, 1976.

\bibitem [2]{ArSi} Arens, Richard \& Singer, I. M, Function values as boundary integrals, {\it Proc. Amer. Math. Soc}  {\bf 5} (1954), 735--745. 

\bibitem [3]{ArSi1}  Arens, Richard \& Singer, I. M, Generalized analytic functions, {\it Trans. Amer. Math. Soc}
{\bf 81}  (1956), 379--393.

\bibitem [4]{Ber} M. Berry, Riemann's zeta function: a model of quantum chaos, in {\it 	Quantum Chaos and Statistical Nuclear Physics}, pp. 1--17, Lecture Notes in Physics {\bf 263}, Springer-Verlag,
New York, 1986.

\bibitem [5]{Be} Besikovitch, A. S., {\it Almost Periodic Functions.} Dover Publications, Inc., New York, 1955.

\bibitem [6]{Bu} Bump, Daniel, {\it Automorphic Forms and Representations.} Cambridge Studies in Advanced Mathematics {\bf 55}. Cambridge University Press, Cambridge, 1997.

\bibitem [7]{Co} Cogdell, J.W., private communication, 27 june 2010.


\bibitem [8]{Con} Connes, A., Noncommutative geometry and the Riemann zeta function. preprint

\bibitem [9]{De} Deligne, P., Les conjectures de Weil I. {\it Inst. Hautes \'{E}tudes 
Sci. Publ. Math. } {\bf 43} (1974), 273--300.


\bibitem [10p]{GVp} Gendron T.M. \& A. Verjovsky, 
Geometric Galois theory, nonlinear number fields and a Galois group interpretation of the idele class group.  {\it Internat. J. Math.}  {\bf 16}  (2005),  no. 6, 567--593.


\bibitem [10e]{GVe} Gendron T.M. \& A. Verjovsky, 
Errata: Geometric Galois theory, nonlinear number fields and a Galois group interpretation of the idele class group. to appear, {\it Internat. J. Math.}


\bibitem [10r]{GV} Gendron T.M. \& A. Verjovsky, 
Geometric Galois theory, nonlinear number fields and a Galois group interpretation of the idele class group.  Revised Version, arXiv:math/0506519.





\bibitem [11]{Gh} Ghys, \'{E}tienne,
Laminations par surfaces de Riemann. in {\it Dynamique et g\'{e}om\'{e}trie complexes (Lyon, 1997)}, pp. 49--95,
{\it Panor. Synth\`{e}ses} {\bf 8}, Soc. Math. France, Paris, 1999. 

\bibitem [12]{GN} Goodman, Arthur \& Newman, D.J., A Wiener type theorem for Dirichlet series.
{\it Proc. AMS}, {\bf 92}, No. 4 (1984), 521--527.


\bibitem [13]{HW} Hewitt, Edwin \& Williamson, J.H., Note on absolutely convergent Dirichlet series.
{\it Proc. AMS}, {\bf 8}, No. 4 (1957), 863--868.


\bibitem [14]{JS} Joyal, Andr\'{e} \& Street, Ross, An introduction to Tannaka duality and quantum groups, in Part II of {\it Category Theory, Proceedings, Como 1990}, eds. A. Carboni, M. C. Pedicchio and G. Rosolini, pp. 411--492, Lectures Notes in Mathematics {\bf 1488}, Springer, Berlin, 1991.

\bibitem [15]{Ka}  Katznelson, Yitzhak, {\it An Introduction to Harmonic Analysis}. Third edition. Cambridge Mathematical Library. Cambridge University Press, Cambridge, 2004.


\bibitem [16]{Kn}  Knopp, Konrad,  {\it Infinite Sequences and Series.}  Dover Publications, Inc., New York, 1956. 


\bibitem [17]{Ko} Koch, {\it Algebraic Number Theory}.  Encyclopedia of Mathematics, Number Theory II, Springer-Verlag, Berlin, 1997.

\bibitem [18]{Kow1} Kowalski, E.
Elementary theory of L-functions I. in {\it An Introduction to the Langlands Program (Jerusalem, 2001)}, pp. 1--20, Birkh\"{a}user, Boston, MA, 2003.

\bibitem [19]{Kow} Kowalski, E.
Classical automorphic forms. in {\it An Introduction to the Langlands Program (Jerusalem, 2001)}, pp. 39--71, Birkh\"{a}user, Boston, MA, 2003. 

\bibitem [20]{Ne} Neukirch, J\"{u}rgen,  {\it Algebraic Number Theory}. 
 Grundlehren der Mathematischen Wissenschaften {\bf 322}. Springer-Verlag, Berlin, 1999.
 
\bibitem [21]{New} Newman, D.J., A simple proof of Wiener's $1/f$ theorem, {\it Proc. Amer. Math. Soc.} {\bf 48} (1975), 264Ð265.

\bibitem [22]{deSh} de Shalit, Ehud,  Artin L-functions.  in {\it An Introduction to the Langlands Program (Jerusalem, 2001)}, pp. 39--71, Birkh\"{a}user Boston, Boston, MA, 2003. 


\bibitem [23]{Su} Sullivan, Dennis,
Bounds, quadratic differentials, and renormalization conjectures. in {\it American Mathematical Society centennial publications, Vol. II (Providence, RI, 1988)}, pp. 417--466, Amer. Math. Soc., Providence, RI, 1992. 

\bibitem [24]{Ta} Tate, J. T.,
Fourier analysis in number fields, and Hecke's zeta-functions. in {\it 1967 Algebraic Number Theory (Proc. Instructional Conf., Brighton, 1965)}, pp. 305--347, Thompson, Washington, D.C. 

\bibitem [25]{Wi} Wigner, Eugene P.,
{\it Group Theory: And its Application to the Quantum Mechanics of Atomic Spectra.}
Pure and Applied Physics. Vol. 5, Academic Press, New York-London 1959.

\end{thebibliography}
\end{document}